\documentclass[11pt,reqno,oneside]{amsart}
\usepackage{amsmath,amsfonts,amssymb,amsthm,fancyhdr,bm,mathtools,enumitem,ifthen}
\usepackage[hyphens]{url}
\usepackage[hidelinks]{hyperref}
\usepackage{comment}
\usepackage[margin=1.1in]{geometry}
\usepackage[dvipsnames]{xcolor}
\usepackage[capitalize]{cleveref}
\usepackage[color=Purple!50!white]{todonotes}
\usepackage[numbers,sort]{natbib}
\usepackage{tikz,ifthen}
\usepackage{mathrsfs,verbatim}
\setuptodonotes{size=\scriptsize}

\newtheorem{theorem}{Theorem}[section]
\newtheorem{lemma}[theorem]{Lemma}
\newtheorem{proposition}[theorem]{Proposition}

\newtheorem{conjecture}[theorem]{Conjecture}

\theoremstyle{definition}

\crefname{equation}{}{}
\crefname{lemma}{Lemma}{Lemmas}
\crefname{proposition}{Proposition}{Propositions}
\crefname{claim}{Claim}{Claims}
\crefname{theorem}{Theorem}{Theorems}
\crefname{conjecture}{Conjecture}{Conjectures}
\crefname{figure}{Figure}{Figures}

\newlist{lemenum}{enumerate}{1}
\setlist[lemenum]{label=(\alph*), ref=\thelemma(\alph*)}
\crefalias{lemenumi}{lemma}

\AddToHook{env/lemma/begin}{\crefalias{theorem}{lemma}}
\AddToHook{env/proposition/begin}{\crefalias{theorem}{proposition}}
\AddToHook{env/corollary/begin}{\crefalias{theorem}{corollary}}
\AddToHook{env/conjecture/begin}{\crefalias{theorem}{conjecture}}
\AddToHook{env/claim/begin}{\crefalias{theorem}{claim}}
\AddToHook{env/question/begin}{\crefalias{theorem}{question}} 
\AddToHook{env/definition/begin}{\crefalias{theorem}{definition}} 
\AddToHook{env/remark/begin}{\crefalias{theorem}{remark}}

\newcommand\margindef[1]{\marginpar{\textcolor{gray!70!black}{\scriptsize #1}}}

\newcommand\ol[1]{\overline{#1}}

\newcommand\wh[1]{\widehat{#1}}

\newcommand\ab[1]{\lvert#1\rvert}

\newcommand{\cei}[1]{\lceil #1 \rceil}

\newcommand{\cE}{\mathcal{E}}

\newcommand{\cH}{\mathcal{H}}

\newcommand{\bP}{\mathbb{P}}
\newcommand{\bE}{\mathbb{E}}
\newcommand{\bR}{\mathbb{R}}

\makeatletter
\NewDocumentCommand{\operator}{%
	m 
	m 
	m 
	m 
}{%
	\expandafter\def\csname #1\endcsname{\expandafter\@ifstar{\csname #1@star\endcsname}{\csname #1@nostar\endcsname}}
	\expandafter\NewDocumentCommand\csname #1@star\endcsname{ O{#2} m }{##1\mathopen{}\left#3##2\right#4\mathclose{}}
	\expandafter\NewDocumentCommand\csname #1@nostar\endcsname{ O{} O{#2} m }{##2\mathopen##1#3##3\mathclose##1#4}
}
\NewDocumentCommand{\coperator}{%
	m 
	m 
	m 
	m 
	m 
}{%
	\expandafter\def\csname #1\endcsname{\expandafter\@ifstar{\csname #1@star\endcsname}{\csname #1@nostar\endcsname}}
	\expandafter\NewDocumentCommand\csname #1@star\endcsname{ O{#2} m m }{##1\mathopen{}\left#3##2\mathrel{}\middle#4\mathrel{}##3\right#5\mathclose{}}
	\expandafter\NewDocumentCommand\csname #1@nostar\endcsname{ O{} O{#2} m m }{##2\mathopen##1#3##3\mathrel##1#4##4\mathclose##1#5}
}
\makeatother

\operator{pr}{\bP}{[}{]}
\operator{ex}{\bE}{[}{]}
\operator{set}{}{\{}{\}}

\coperator{cpr}{\bP}{[}{|}{]}
\coperator{cex}{\bE}{[}{|}{]}

\coperator{cset}{}{\{}{\colonDelimiter}{\}}

\DeclarePairedDelimiter{\abs}{\lvert}{\rvert}
\DeclarePairedDelimiter{\paren}{\lparen}{\rparen}

\newcommand{\comp}{\textup{c}}
\newcommand{\su}{\subseteq}
\newcommand{\eps}{\varepsilon}

\let\leq\leqslant
\let\geq\geqslant
\let\le\leqslant
\let\ge\geqslant

\DeclareMathDelimiter{\colonDelimiter}{\mathord}{operators}{"3A}{operators}{"3A}

\title{Disproof of the odd Hadwiger conjecture}
\author{Marcus K\"uhn \and Lisa Sauermann \and Raphael Steiner \and Yuval Wigderson}
\thanks{MK: Institute for Applied Mathematics, University of Bonn, Germany. \texttt{marcus.kuehn@uni-bonn.de}.}
\thanks{LS: Institute for Applied Mathematics, University of Bonn, Germany. \texttt{sauermann@iam.uni-bonn.de}. Research  supported by the
DFG Heisenberg Program}
\thanks{RS: Department of Mathematics, ETH Z\"urich, Switzerland. \texttt{raphaelmario.steiner@math.ethz.ch}. Research supported by SNSF Ambizione grant 216071}
\thanks{YW: Institute for Theoretical Studies, ETH Z\"urich, Switzerland. \texttt{yuval.wigderson@eth-its.ethz.ch}. Research supported by Dr.\ Max R\"ossler, the Walter Haefner Foundation, and the ETH Z\"urich Foundation}
\date{}

\begin{document}

\begin{abstract}
	We prove that there exist graphs which do not contain $K_t$ as an odd minor and whose chromatic number is at least $(\frac 32-o(1))t$. This disproves, in a strong form, the odd Hadwiger conjecture of Gerards and Seymour from 1993. 
\end{abstract}
\maketitle

\section{Introduction}

Hadwiger's conjecture \cite{MR12237}, formulated in 1943, is one of the central open problems in graph theory. Recall that a graph $F$ is a \emph{minor} of a graph $G$ if a graph isomorphic to $F$ can be obtained from $G$ by a sequence of vertex deletions, edge deletions, and edge contractions. 
\begin{conjecture}[Hadwiger's conjecture]
Let $G$ be a graph. If $\chi(G)\geq t$, then $G$ has $K_t$ as a minor. Equivalently, if $G$ does not have $K_t$ as a minor, then $\chi(G)<t$. 
\end{conjecture}
This is a vast generalization of the Four-Color Theorem \cite{MR543793,MR543792}, which follows immediately from the $t=5$ case of Hadwiger's conjecture. Moreover, Hadwiger's conjecture would determine the maximum chromatic number of any minor-closed family of graphs, thus unifying and generalizing a wide array of different results and conjectures bounding  the chromatic number of various graph classes, see e.g.\ \cite{MR228378,MR831801,MR1070462}.

Despite decades of effort, Hadwiger's conjecture remains extremely far from being resolved. It is only known to hold \cite{MR1238823} for $t \leq 6$, with the cases $t=5$ and $t=6$ both shown to be equivalent to the Four-Color theorem, and even the $t=7$ case appears to be out of reach at the moment. Many researchers have attempted to prove weakenings of the conjecture, such as finding functions $f(t)$ such that every graph $G$ not having $K_t$ as a minor satisfies $\chi(G)<f(t)$; Hadwiger's conjecture is the assertion that this holds for $f(t)=t$. There has been a long line of works \cite{MR779891,MR735367,MR713722,MR4576840,MR4866587} steadily improving the function $f(t)$; the current record is due to Delcourt and Postle \cite{MR4868948}, who showed that this holds for $f(t)=Ct\log \log t$, where $C$ is an absolute constant. 
For more on the history and progress towards Hadwiger's conjecture, we refer to the surveys \cite{MR4680416,MR3526944,MR1411244}.

In this paper, our focus will be on a strengthening of Hadwiger's conjecture proposed by Gerards and Seymour in 1993 (see \cite[Section 6.5]{MR1304254}, \cite[Section 7]{MR3526944}), known as the \emph{odd Hadwiger conjecture}. To explain the motivation behind this strengthening, let us go back to considering the cases of Hadwiger's conjecture when $t$ is small. For $t=3$, it asserts that if $\chi(G)\geq 3$, then $G$ contains $K_3$ as a minor. This statement is true, since containing $K_3$ as a minor is equivalent to containing a cycle. But evidently, this true statement is much weaker than the ``real'' statement here, which is that if $\chi(G)\geq 3$, then $G$ contains an \emph{odd} cycle. The next case $t=4$ of Hadwiger's conjecture is equivalent to the fact that every graph $G$ with $\chi(G)\geq 4$ contains a $K_4$-subdivision as a subgraph. And also here, a much stronger statement involving parity is known to be true: Catlin~\cite{MR532593} proved that every such graph contains a $K_4$-subdivision where every cycle formed by three subdivision paths is of odd length. Moreover, an even stronger statement is still true: Toft conjectured in 1975, and Thomassen~\cite{MR1848060} and Zang~\cite{MR1631313} independently proved, that every graph $G$ with $\chi(G)\geq 4$ even contains a $K_4$-subdivision where \emph{all} subdivision paths have odd length. Motivated by these results and attempting to generalize Catlin's result, Gerards and Seymour proposed a natural notion, stricter than that of containing $K_t$ as a minor, which respects the parities of cycles, and conjectured that even this stronger property is implied by large chromatic number.

More precisely,
let $G$ be a graph. We say that $G$ \emph{has $K_t$ as an odd minor} if there exist vertex-disjoint trees $T_1,\dots,T_t\subseteq G$, as well as functions $\chi_i:V(T_i) \to \{\text{black},\text{white}\}$, with the following properties:
\begin{itemize}
	\item For every $i$, $\chi_i$ is a proper $2$-coloring of $T_i$, that is, for all $\set{x,y} \in E(T_i)$, we have that $\chi_i(x)\neq \chi_i(y)$.
	\item For all $i \neq j$, there are vertices $x_i \in V(T_i), x_j \in V(T_j)$ such that $\set{x_i, x_j} \in E(G)$ and $\chi_i(x_i)=\chi_j(x_j)$. 
\end{itemize}
With this definition in place, we can state the odd Hadwiger conjecture of Gerards and Seymour.
\begin{conjecture}[Odd Hadwiger conjecture]\label{conj:odd hadwiger}
	Let $G$ be a graph.
If $\chi(G)\geq t$, then $G$ has $K_t$ as an odd minor. Equivalently, if $G$ does not have $K_t$ as an odd minor, then $\chi(G)<t$. 
\end{conjecture}

As discussed above, the $t=3$ case of \cref{conj:odd hadwiger} is quite straightforward, and is equivalent to the statement that every non-bipartite graph has an odd cycle. The $t=4$ case is equivalent to the aforementioned result of Catlin. The $t=5$ case of \cref{conj:odd hadwiger} is a substantial strengthening of the Four-Color theorem, and a proof was announced by Guenin \cite{MR2164877}, but the conjecture remains wide open for $t\geq 6$.

Similarly to the case of Hadwiger's conjecture, there have been a large number of results \cite{MR698708,MR2467815,MR2518204,MR4426767,2010.05999,fracoddhadwiger,oddclustered1,oddclustered2,oddclustered3,clusteredodd,topoddhadwiger,topoddhadwiger2} proving weakenings of the odd Hadwiger conjecture. Interestingly, the density-based techniques used in the study of Hadwiger's conjecture do not work without serious modifications in the setting of odd minors, as they do not interact naturally with the parity requirements defining an odd minor. Despite this fact, the third author proved~\cite{MR4379303} that Hadwiger's conjecture and its odd variant are asymptotically equivalent. More precisely, he showed that, if $f$ is any function such that any graph $G$ not containing $K_t$ as a minor satisfies $\chi(G)< f(t)$, then any graph $G$ not containing $K_t$ as an odd minor satisfies $\chi(G) < 2\cdot f(t)$. In particular, combined with the result of Delcourt--Postle \cite{MR4868948} mentioned above, this theorem shows that every graph $G$ with $\chi(G)\geq Ct\log \log t$ contains $K_t$ as an odd minor. But more importantly, this result shows that Hadwiger's conjecture and its odd variant are very closely connected; for instance, if Hadwiger's conjecture is true, then every graph $G$ with $\chi(G) \geq 2t$ contains $K_t$ as an odd minor.

Despite all of this evidence towards the odd Hadwiger conjecture, it turns out to be false. Our main result disproves the conjecture in the following strong form.

\begin{theorem}\label{thm:main}
	For any fixed $\delta>0$ and every integer $t$ which is sufficiently large with respect to $\delta$, there is a graph $G$ with $\chi(G) \geq (\frac 32 -\delta)t$ which does not have $K_t$ as an odd minor. In particular, for every sufficiently large $t$, there is a graph with $\chi(G)\geq t$ which does not have $K_t$ as an odd minor.
\end{theorem}
Thus, not only is the odd Hadwiger conjecture false, it is false by a factor of roughly $\frac 32$ in the sense that the chromatic number of $G$ can exceed the order of the largest odd clique minor by an asymptotic factor of $\frac 32$. Note that, as discussed above, Hadwiger's conjecture would imply that this ratio is at most $2$. Unfortunately, as we shortly discuss further, the constant $\frac 32$ is a hard limit of our approach.

Our graph $G$ in \cref{thm:main} forming a counterexample to the odd Hadwiger conjecture is a graph with independence number $\alpha(G)=2$. There has been a lot of previous research on Hadwiger's conjecture and its odd variant for graphs of independence number $2$. Note that if $G$ is such a graph, then $\chi(G)\geq \cei{n/\alpha(G)}=\cei{n/2}$, so Hadwiger's conjecture (resp.\ \cref{conj:odd hadwiger}) predicts that $G$ contains a clique minor (resp.\ odd clique minor) of order $\cei{n/2}$. A simple argument of Duchet and Meyniel \cite{MR671905} shows that if $G$ is an $n$-vertex graph with $\alpha(G)=2$, then $G$ contains a clique minor of order $\lceil n/3\rceil$. This was extended to the odd minor setting by Kawarabayashi and Song \cite{kawasong}, who proved that such a graph $G$ must in fact have $K_{\lceil n/3\rceil}$ as an odd minor.
It remains a major open problem (e.g.\ \cite[Open problem 4.8]{MR3526944} and~\cite{MR2674817,2512.17114}) to improve the constant factor $\frac 13$ in either of these results; see \cite[Section 4]{MR2735925} for the current best known result on Hadwiger's conjecture in graphs of independence number $2$, improving the Duchet--Meyniel bound by a sublinear amount. Our counterexample to \cref{conj:odd hadwiger} shows that the bound of $\lceil n/3\rceil$ due to Kawarabayashi--Song for the odd Hadwiger conjecture in graphs of independence number $2$ is asymptotically best possible. More precisely, we prove the following strengthening of \cref{thm:main}.

\begin{theorem}\label{thm:main alpha=2}
	For any fixed $\eta>0$, and every integer $n$ which is sufficiently large with respect to $\eta$, there exists an $n$-vertex graph $G$ with $\alpha(G)\leq 2$ which does not have $K_t$ as an odd minor for any $t\geq (\frac 13 + \eta)n$. 
\end{theorem}
This theorem directly implies \cref{thm:main}, using that any graph $G$ with $\alpha(G)\leq 2$ satisfies $\chi(G)\geq n/\alpha(G)\geq n/2$. More details on this simple deduction can be found in \cref{sec:deduce main thm}.

The factor $\frac 13$ in \cref{thm:main alpha=2} is optimal, due to the aforementioned result of Kawarabayashi and Song. We complement \cref{thm:main,thm:main alpha=2} with the following result, proved in \cref{appendix:optimality}, showing that also the factor $\frac 32$ for the chromatic number in \cref{thm:main} is optimal for graphs $G$ with $\alpha(G)\leq 2$.

\begin{proposition}\label{prop:stupid}
    For any positive integer $t$ and any graph $G$ with $\alpha(G)\leq 2$ which does not have $K_t$ as an odd minor, we have $\chi(G)\leq \lceil\frac{3}{2}(t-1)\rceil$. 
\end{proposition}

\section{Heuristics: pseudorandom triangle-free graphs}
We now turn to discuss, at a high level, the heuristics and intuition  underlying our construction of the graph $G$ in \cref{thm:main alpha=2}. This section can easily be skipped by readers who prefer to proceed directly to the formal construction in \cref{sec:construction} and the proofs in the rest of the paper. 

Since our goal is to build a graph $G$ with $\alpha(G)\leq 2$, it is more convenient to work in the complement: we set $G=\ol H$, and wish to construct a triangle-free graph $H$ whose complement does not contain $K_t$ as an odd minor. 

The simplest idea one could try is to let $H$ be a random graph. Recall that the \emph{binomial random graph} $G(n,p)$ is an $n$-vertex graph in which every unordered pair of distinct vertices forms an edge independently with probability $p$. As our first idea, we could sample $H \sim G(n,p)$, and aim to prove that with positive probability, we simultaneously have that $H$ is triangle-free, and that $G=\ol H$ contains no large odd minor. For this discussion, we will only try to show that $G$ does not contain $K_{n/2}$ as an odd minor---which already suffices to disprove \cref{conj:odd hadwiger}---as opposed to trying to obtain the optimal bound on the odd minor order given in \cref{thm:main alpha=2}. For simplicity, let us assume that $n$ is even.

If $G$ has $K_{n/2}$ as an odd minor, then there exist trees $T_1,\dots,T_{n/2}$ satisfying the definition of an odd minor. Note that since we have $n/2$ trees on $n$ vertices, the average number of vertices of a tree $T_i$ is at most $2$. For simplicity, let us pretend for the moment that all these trees have exactly $2$ vertices (unsurprisingly, this special case is the hardest one). The structure we have in this case is called an \emph{odd connected matching} (of size $n/2$) in the literature. Studying the presence of odd connected matchings in $G$, the complement of a random graph, is somewhat annoying. It will be simpler and still sufficient for us to consider an even weaker notion,
which we call an \emph{odd connected pairing}, where we do not require $T_i$ to be a tree. More precisely, we define an \emph{odd connected pairing of size $s$} in a graph $G$ to be a set $\{(u_1,v_1),\ldots,(u_s,v_s)\}$ consisting of $s$ ordered pairs where $u_1,v_1,u_2,v_2,\dots,u_s,v_s \in V(G)$ are pairwise distinct vertices, and such that for all $i \neq j$, at least one of $\set{u_i, u_j}$ or $\set{v_i, v_j}$ is an edge of $G$. Note that if we further insisted that $\set{u_i, v_i}$ is an edge for all $i$, this would precisely correspond to having $K_s$ as an odd minor with all the trees $T_i$ consisting of exactly $2$ vertices. 

We can now begin in earnest to analyze our random construction $G = \ol H$, where $H \sim G(n,p)$. To get started, let us first estimate the probability that $G$ contains an odd connected pairing of size $n/2$. For any fixed pairing $(u_1,v_1),\dots,(u_{n/2},v_{n/2})$, and any fixed indices $i,j$, the probability that at least one of $\{u_i,u_j\}$ and $\{v_i,v_j\}$ is an edge of $G$ is $1-p^2$, since with probability $p^2$ both are included as edges of $H$. Since these events for different pairs $\{i,j\}$ of indices are independent, the probability that the fixed pairing forms an odd connected pairing in $G$ is precisely $(1-p^2)^{\binom{n/2}2}$, since we have $\binom{n/2}2$ pairs $\{i,j\}$ of indices. We now want to apply a union bound over all choices of the pairing; the number of such pairings equals the number of perfect matchings in $K_n$, which is given by $(n-1)(n-3)\dotsm 1$. Asymptotically, this quantity is equal to $(n/e+o(n))^{n/2} = \exp((\frac 12 +o(1))n\log n)$. Applying the union bound, we see that
\begin{align*}
\pr{G\text{ contains an odd connected pairing}&\text{ of size }\tfrac n2}\leq (n-1)(n-3)\dotsm 1 \cdot (1-p^2)^{\binom{n/2}2}\\
&\leq \exp\left(\Big(\frac 12 + o(1)\Big)n\log n-p^2 \binom{n/2}2\right).
\end{align*}
In particular, if $\frac 18 p^2 n^2 > (\frac 12+o(1))  n \log n$, then this exponent is negative for large $n$, and with high probability\footnote{Throughout the paper, we say an event occurs \emph{with high probability}\ if its probability tends to $1$ as $n\to \infty$.}, $G$ does not contain an odd connected pairing of size $n/2$. The condition $\frac 18 p^2 n^2 > (\frac 12+o(1)) n \log n$ is equivalent to $p>(2+o(1))\sqrt{\log n/n}$.

We also want $H$ to be triangle-free, so that $\alpha(G)\leq 2$. Here we encounter some bad news, since it is well-known that $G(n,p)$ is with high probability triangle-free if and only if $p = o(1/n)$, much smaller than the requirement $p>(2+o(1))\sqrt{\log n/n}$ we found above. However, we are empowered to modify $H$ after sampling from $G(n,p)$, so long as these modifications are ``insignificant'' enough to not seriously affect the probability that $G=\ol H$ contains an odd connected pairing. It is also well-known that, so long as $p = o(1/\sqrt n)$, a typical edge of $G(n,p)$ does not lie on any triangle; this means, intuitively, that one can delete all triangles from $G(n,p)$ without significantly affecting its structure so long as $p = o(1/\sqrt n)$. In fact, by being very careful, we could even hope to do this for $p \leq c/\sqrt n$, where $c>0$ is a small absolute constant.

Unfortunately, this is still not good enough: our requirements $p > (2+o(1))\sqrt{\log n/n}$ and $p \leq c/\sqrt n$ remain incompatible. In order to explain how we overcome this, we briefly take a detour to discuss recent research on a seemingly unrelated question, which nevertheless holds the key to proving \cref{thm:main alpha=2}.

This seemingly unrelated question is about the asymptotics of the off-diagonal Ramsey number $r(3,k)$, which is defined as the minimum $n$ such that every $n$-vertex graph contains either a triangle or an independent set of order $k$. Equivalently, we can ask for the minimum independence number among all $n$-vertex triangle-free graphs. For many decades, the best known upper bounds for this problem were obtained by a careful analysis of the random graph $G(n,c/\sqrt n)$. Essentially, by following the intuition described above, one can carefully delete all triangles from this random graph without significantly affecting its independence number. Different arguments along these lines were developed by Erd\H os \cite{MR120168}, Spencer \cite{MR491337}, Erd\H os--Suen--Winkler \cite{MR1370965}, Bollob\'as \cite[Chapter 12]{MR1864966}, and Krivelevich \cite{MR1369060}, among others. Finally, a major breakthrough on this problem was obtained by Kim \cite{MR1369063}, who designed a different randomized procedure for generating a triangle-free graph, and proved that with high probability, this process produces an $n$-vertex triangle-free graph with independence number $O(\sqrt{n \log n})$, matching the classical lower bound of Ajtai--Koml\'os--Szemer\'edi \cite{MR600598} and Shearer \cite{MR708165} up to a constant factor. Subsequently, Bohman \cite{MR2522430} analyzed a simpler random process---the so called \emph{triangle-free process}---and proved that with high probability, it generates a pseudorandom triangle-free graph with edge density $\Omega(\sqrt{\log n/n})$. Substantially more refined results were obtained independently by Bohman and Keevash \cite{MR4201797} and by Fiz Pontiveros, Griffiths, and Morris \cite{MR4073152}, who proved that the final edge density of the triangle-free process is $p=(\frac 12-o(1))\sqrt{\log n/n}$; moreover, they showed that in various ways, the triangle-free process is nearly indistinguishable from the binomial random graph $G(n,p)$ of the same edge density, apart from the fact that it has no triangles.

These results seem almost good enough for us: they give a pseudorandom triangle-free graph of edge density $(\frac 12-o(1))\sqrt{\log n/n}$, whereas we would like edge density $(2+o(1))\sqrt{\log n/n}$, merely a factor of $4$ away. Unfortunately, this ``mere'' factor of $4$ is very stubborn: the results of \cite{MR4201797,MR4073152} also show, essentially, that it is impossible to create a denser triangle-free graph without sacrificing some of the pseudorandomness properties of the triangle-free process. This intuition led the authors of \cite{MR4073152} to conjecture that in fact, the triangle-free process has the asymptotically smallest independence number among all $n$-vertex triangle-free graphs.

However, in a recent breakthrough, Campos, Jenssen, Michelen, and Sahasrabudhe \cite{2505.13371} disproved this conjecture. One of their key insights is that for the $r(3,k)$ problem, there is no need to require the very strong pseudorandomness conditions enjoyed by the triangle-free process; one can sacrifice a small amount of pseudorandomness in order to obtain higher edge density, and this trade-off can be beneficial. This was followed by another breakthrough of Hefty, Horn, King, and Pfender \cite{2510.19718}, who found a much simpler randomized construction yielding pseudorandom triangle-free graphs of higher edge density. For their application, they select edge density $p=(1+o(1))\sqrt{\log n/n}$ (which is expected to be optimal for the application to $r(3,k)$), but their construction naturally works at a wide range of possible edge densities.
Indeed, in follow-up work on a closely related problem, Campos, Jenssen, Michelen, Pfender, and Sahasrabudhe \cite{2511.10641} used essentially the same construction, but crucially worked with a polynomially higher edge density.

These results are the starting point of our approach, as they resolve the central incompatibility discussed above: we now have access to a construction of pseudorandom triangle-free graphs whose edge density can be larger than the threshold $(2+o(1))\sqrt{\log n/n}$ we encountered above. In fact, in order to simplify the proofs of our main technical results, it will be convenient to work with a logarithmically higher edge density, namely $p = \Theta((\log n)^2/\sqrt n)$.

The basic idea in the works \cite{2505.13371,2511.10641,2510.19718} is the following simple approach. Let $H_{\mathrm{start}}$ be a triangle-free graph on $m$ vertices, for some $m$ that is smaller but not much smaller than $n$, and let $H_{\mathrm{blowup}}$ be obtained from $H_{\mathrm{start}}$ by blowing up each vertex by a factor of $n/m$. That is, we replace each vertex of $H_{\mathrm{start}}$ by an independent set of order $n/m$, and each edge of $H_{\mathrm{start}}$ by a complete bipartite graph between the independent sets corresponding to its endpoints. Then on the one hand, if $H_{\mathrm{start}}$ is pseudorandom in an appropriate sense, then $H_{\mathrm{blowup}}$ retains many of the pseudorandomness properties; for example, the edge distribution among all large sets in $H_{\mathrm{blowup}}$ is roughly the same as it was in $H_{\mathrm{start}}$. Moreover, the two graphs $H_{\mathrm{blowup}},H_{\mathrm{start}}$ have essentially the same edge density. However, when viewed as a function of $n$, the edge density of $H_{\mathrm{blowup}}$ is higher than that of $H_{\mathrm{start}}$. For example, if $H_{\mathrm{start}}$ had $m^{3/2}$ edges, then $H_{\mathrm{blowup}}$ has $m^{3/2}\cdot(n/m)^2 = n^2/\sqrt m$ edges, and, if $m=o(n)$, then $n^2/\sqrt m = \omega(n^{3/2})$. In particular, if we take $H_{\mathrm{start}} \sim G(m,m^{-1/2})$, then $H_{\mathrm{blowup}}$ will be an $n$-vertex graph with edge density higher than $n^{-1/2}$, which still maintains some of the pseudorandomness properties of $G(m,m^{-1/2})$.

Of course, for many applications, including the application to $r(3,k)$ and to the odd Hadwiger conjecture, we have done essentially nothing: 
although we have artificially increased the relative edge density, the large independent sets in $H_{\mathrm{blowup}}$ completely counteract whatever we gained from this.
The other new idea in \cite{2505.13371,2511.10641,2510.19718} is to ``conceal'' the blowup structure of $H_{\mathrm{start}}$ by overlaying it with another, independent source of (pseudo)randomness. In \cite{2505.13371} this is accomplished by running the triangle-free process on the blown-up graph $H_{\mathrm{blowup}}$, but the much simpler approach taken in \cite{2510.19718,2511.10641} is simply to randomly overlay two blowups. Because the overlaying is done in a random way, the blowup structures of the two graphs are essentially ``orthogonal'', and the edges of one blowup destroy the structure of the other blowup. This is also the approach we take.

More specifically, we start with two binomial random graphs with distribution $G(m,p)$, where $p=o(1/\sqrt m)$, but $p = \omega(\sqrt{\log n/n})$. To make them triangle-free, in each of these graphs we first delete all edges occurring in triangles. Since $p=o(1/\sqrt m)$, we expect a negligible proportion of these edges to be deleted this way. Thus, we expect each of the resulting random graphs to have roughly $pm^2/2$ edges. We now blow each of these graphs up by a factor of $n/m$ and randomly overlay these blowups (for more details on this step, see the next section). This way, we obtain an $n$-vertex graph with roughly $2 \cdot (pm^2/2) \cdot (n/m)^2 = pn^2$ edges. Again, we expect few edges to participate in triangles consisting of edges from both of the blow-up graphs. We then delete an edge from every such triangle in a certain way, so that at the end of this process we expect to end up with a triangle-free graph $H$ of edge density roughly $p$; crucially, we have that $p = \omega(\sqrt{\log n/n})$, and thus our heuristic from the binomial random graph suggests that $\ol H$ is unlikely to contain an odd connected pairing of size $n/2$.

\section{The construction}\label{sec:construction}

Let $n$ be a large integer. We fix\footnote{We remark that there is a lot of flexibility in these choices; we really only need that $n/m$ is at least some small fixed power of $\log n$ and at most some small fixed power of $n$, and that the same holds for $p^2 n$ and $(p^2 m)^{-1}$.} parameters\margindef{$n,m,p$}
\[
	m = \left\lceil\frac{n}{(\log n)^{8}}\right\rceil\qquad\text{and} \qquad p = \frac{1}{m^{1/2}\cdot (\log m)^2}.
\]
Our construction is based on the following independent random choices: 
\begin{itemize}
	\item a binomial random graph $H_R \sim G(m,p)$ with vertex set $V_R=[m]$,\margindef{$H_R$}
	\item a binomial random graph $H_B \sim G(m,p)$ with vertex set $V_B=[m]$, and\margindef{$H_B$}
	\item two independent, uniformly random functions $\pi_R:[n] \to [m]$ and $\pi_B:[n] \to [m]$.\margindef{$\pi_R,\pi_B$} Here, we think of $\pi_R$ as a function to $V_R$ and $\pi_B$ as a function to $V_B$.
\end{itemize}
We now let $H_R'$ be obtained from $H_R$ by deleting every edge that lies in a triangle, and similarly for $H_B'$\margindef{$H_R',H_B',H_0$}. We then define a graph $H_0$ with vertex set $[n]$ by taking the union of the pull-backs of $H_R'$ and $H_B'$ along the maps $\pi_R$ and $\pi_B$, respectively. Formally, we join distinct $u,v \in [n]$ by an edge of $H_0$ if
\[
	\set{\pi_R(u),\pi_R(v)} \in E(H_R') \qquad\text{ or }\qquad \set{\pi_B(u),\pi_B(v)} \in E(H_B').
\]
For an edge $\set{u,v} \in E(H_0)$, we call the edge $\set{u,v}$ \emph{red} if $\set{\pi_R(u),\pi_R(v)} \in E(H_R')$, and we call the edge $\set{u,v}$ \emph{blue} if $\set{\pi_B(u),\pi_B(v)} \in E(H_B')$. Note that the graph of red edges is precisely a blowup of $H_R'$, where vertex $x \in V_R$ is replaced by an independent set of size $\ab{\pi_R^{-1}(x)}$, and similarly for the blue edges. Also note that some edges may receive both colors. On the other hand, since $H_R',H_B'$ are triangle-free graphs, there is no triangle in $H_0$ all of whose edges receive the same color. 

Thus, every triangle in $H_0$ must comprise two red edges and one blue edge, or two blue edges and one red edge. We then define $H$\margindef{$H$} to be the subgraph of $H_0$ obtained by deleting the edge of the minority color from each triangle. That is, for every red edge $\set{u,v} \in E(H_0)$ for which there exists $w$ such that $\set{u,w},\set{v,w} \in E(H_0)$ are blue edges, we delete $\set{u,v}$ from $H_0$ in constructing $H$; and similarly for every blue edge $\set{u,v} \in E(H_0)$ for which there exists $w$ such that $\set{u,w},\set{v,w} \in E(H_0)$ are red edges. It is not relevant for what follows, but it may be helpful to note that every deleted edge has only one color, since there are no monochromatic triangles in $H_0$. 

At the end of this process, we have constructed an $n$-vertex triangle-free graph $H$. We conclude the construction by taking $G=\overline H$ to be the complement of $H$, so that $\alpha(G)\leq 2$. As mentioned in the previous section, an \emph{odd connected pairing of size $s$}\margindef{odd con-\\nected pairing} in a graph $G$ is defined to be a set $\{(u_1,v_1),\ldots,(u_s,v_s)\}$ consisting of $s$ ordered pairs where $u_1,v_1,u_2,v_2,\dots,u_s,v_s \in V(G)$ are pairwise distinct vertices, and such that for all $i \neq j$, at least one of $\set{u_i, u_j}$ or $\set{v_i, v_j}$ is an edge of $G$. We will prove that with high probability, $G$ contains no odd connected pairing of size at least $\varepsilon n$, which then implies the following theorem. 

\begin{theorem}\label{thm:odd pairing}
	For every fixed $\eps>0$ and all sufficiently large $n$, there exists an $n$-vertex graph $G$ with $\alpha(G)\leq 2$ which does not contain an odd connected pairing of size at least $\eps n$. 
\end{theorem}
It is not hard to show, as we do in \cref{sec:deduce main thm}, that \cref{thm:odd pairing} implies \cref{thm:main alpha=2}, and hence \cref{thm:main}. Thus, for the majority of the rest of the paper, we focus on proving \cref{thm:odd pairing}.

\subsection*{Organization and notation}
In \cref{sec:probabilistic lemmas}, we state the main probabilistic lemmas we will need to analyze our construction. We then prove \cref{thm:odd pairing} using these lemmas in \cref{sec:main proof} and afterwards deduce \cref{thm:main,thm:main alpha=2}. We prove the probabilistic lemmas in \cref{sec:coupling,sec:connection survives}, and end in \cref{sec:conclusion} with some concluding remarks. Finally, we prove \cref{prop:stupid} in \cref{appendix:optimality}.

For a non-negative integer $n$, we write $[n]:=\{1,\dots,n\}$, as usual. All logarithms are to base~$e$. To save parentheses, we write $\log^k n$ instead of $(\log n)^k$. Given a set $S$ and an integer $k$, we use $\binom Sk$ to denote the collection of all $k$-element subsets of $S$. For an event $\mathcal{E}$, we write $\mathcal{E}^\comp$ to denote the complementary event. We furthermore use the standard asymptotic notation $O,\Omega,o$, and $\omega$. For a graph $G$, we denote by $\Delta_G$ its maximum degree, by $e(G)$ its number of edges, by $\chi(G)$ its chromatic number, by $\alpha(G)$ its independence number, and by $\nu(G)$ its matching number.

In order to help the reader track the main definitions, we include marginal notes to indicate the first time a key notion is introduced.

\section{Probabilistic lemmas}\label{sec:probabilistic lemmas}
In order to prove \cref{thm:odd pairing}, we will analyze a number of properties depending on the random choices of $H_R,H_B,\pi_R,\pi_B$. In this section, we introduce various events defined by these random choices, and state a number of lemmas which say, roughly, that various ``good'' outcomes happen with high probability. In the next section, we use these probabilistic lemmas to prove \cref{thm:odd pairing}. We now introduce each of these events in turn.

The first four are quite straightforward: they say that the random graphs $H_R,H_B$ and the random functions $\pi_R,\pi_B$ are ``balanced'', in the sense that the graphs do not have abnormally high maximum degree, and the functions do not have abnormally large fibers. Formally, we define $\mathcal D_R$ to be the event\margindef{{$\mathcal D_R, \mathcal D_B$}} that the maximum degree of $H_R$ is at most $2pm$, and similarly $\mathcal D_B$ to be the event that the maximum degree of $H_B$ is at most $2pm$. We also define $\mathcal F_R$ to be the event that $\ab{\pi_R^{-1}(x)}\leq 2n/m$ for all $x \in V_R$, and\margindef{{$\mathcal F_R, \mathcal F_B$}} similarly $\mathcal F_B$ to be the event that $\ab{\pi_B^{-1}(x)}\leq 2n/m$ for all $x \in V_B$. Note that the events $\mathcal D_R, \mathcal D_B, \mathcal F_R$ and $\mathcal F_B$ only depend on the outcome of the random object $H_R, H_B, \pi_R$, and $\pi_B$, respectively. The following simple lemma, an elementary consequence of the Chernoff bound, shows that these four events hold with high probability (the proof can be found in \cref{subsec:concentration ineqs}).
\begin{lemma}\label{lem:easy chernoff}
	The events $\mathcal D_R, \mathcal D_B, \mathcal F_R, \mathcal F_B$ hold with high probability as $n \to \infty$.
\end{lemma}

In our next condition, we aim to control the interactions of the two functions $\pi_R$ and $\pi_B$. The point here is that these maps should be ``orthogonal'', so that the randomness from $H_R$ and from $H_B$ interact a great deal in defining $H$. However, the formal condition requires some setup. 
For $s\in\mathbb{N}$, an \emph{$s$-pairing}\margindef{{$s$-pairing}} in $[n]$ is a collection of pairs $\{(u_1,v_1),\dots,(u_{s},v_{s})\}$, where $u_1,\dots,u_{s},v_1,\dots,v_{s}$ are distinct elements of $[n]$. 
For an $s$-pairing $P=\{(u_1,v_1),\dots,(u_{s},v_{s})\}$ in $[n]$, and $\eps>0$, we say that a map $\pi:[n] \to [m]$ is \emph{$\varepsilon$-respectful}\margindef{{$\varepsilon$-respectful}} of $P$ if there is a subset $I \subseteq [s]$ with $\ab I \geq \varepsilon s$ such that $\pi(u_i)\neq \pi(v_i)$ for all $i \in I$, and furthermore the pairs $\{\pi(u_i),\pi(v_i)\}$ are distinct for all $ i \in I$. That is, $\pi$ is $\varepsilon$-respectful of the pairing $P$ if at least $\varepsilon s$ of the pairs in $P$ get mapped to distinct pairs in $\binom{[m]}2$. Finally, for $\eps>0$, let $\mathcal E(\varepsilon)$ be the event\margindef{{$\cE(\varepsilon)$}} that for each $s$-pairing $P$ in $[n]$ with $s= \lceil\eps n\rceil$, at least one of the random functions $\pi_R$ and $\pi_B$ is $\varepsilon$-respectful of $P$. Note that this event only depends on the random functions $\pi_R$ and $\pi_B$, but not on the random graphs $H_R$ and $H_B$. The following lemma, proved in \cref{sec:coupling}, shows that the event $\cE(\varepsilon)$ happens with high probability.

\begin{lemma}\label{lem:random coupling}
	Fix $0<\varepsilon<\frac 1{10}$. Then the event $\cE(\varepsilon)$ holds with high probability as $n \to \infty$.
\end{lemma}

Recall that the graph $H'_R$ is obtained from $H_R$ by deleting all edges in triangles, and $H'_B$ is obtained analogously from $H_B$. Note that in $H_R$, the expected number of triangles is much less than the expected number of edges, hence a typical edge is unlikely to be on a triangle. As a consequence, we expect few edges to be deleted when passing from $H_R$ to $H_R'$. However, for different edges, there are also dependencies in the probabilities of these edges being on triangles and hence being deleted when passing from $H_R$ to $H_R'$. In our probabilistic analysis, we therefore first consider a somewhat denser random graph $H^*_R$, and then sample $H_R$ by randomly sparsifying $H^*_R$, selecting edges of $H^*_R$ with an appropriate probability. Crucially, in this intermediate random graph $H^*_R$, with high probability every edge is contained in only very few triangles. This is helpful, since we then have much fewer dependencies for different edges to lie on triangles, when further sparsifying $H^*_R$ to obtain $H_R$.

Formally, we consider independent binomial random graphs\margindef{{$H^*_R, H^*_B$}} $H^*_R\sim G(m,m^{-1/2})$ with vertex set $V_R=[m]$ and $H^*_B\sim G(m,m^{-1/2})$ with vertex set $V_B=[m]$ (also independent from $\pi_R$ and $\pi_B$), where $H^*_R$ is coupled with $H_R$ such that $H_R\su H^*_R$ always holds, and $H^*_B$ is coupled with $H_B$ such that $H_B\su H^*_B$ always holds. We furthermore choose these couplings in such a way that, conditioning on an outcome $\wh{H_R^*}$ of $H_R^*$, the distribution of the random graph $H_R$ is given by taking every edge of $\wh{H_R^*}$ into $H_R$ independently at random with probability $p/m^{-1/2}=1/{\log^{2}m}$ (also independent from all other random choices related to $\pi_R$, $\pi_B$, $H^*_B$, and $H_B$); and analogously for $H_B$ when conditioning on an outcome $\wh{H_B^*}$ of $H_B^*$.

Now, let \margindef{{$\mathcal T_R, \mathcal T_B$}} $\mathcal T_R$ be the event that every edge $e\in E(H_R^*)$ is contained in at most $20\log m$ triangles in $H_R^*$. Analogously, let $\mathcal T_B$ be the event that every edge $e\in E(H_B^*)$ is contained in at most $20\log m$ triangles in $H_B^*$. It is not hard to show that these events $\mathcal T_R$ and  $\mathcal T_B$ hold with high probability; this again follows from an easy Chernoff bound argument (see \cref{subsec:concentration ineqs} for the details).

\begin{lemma}\label{lem:easy chernoff triangles}
	The events $\mathcal T_R$ and $\mathcal T_B$ hold with high probability as $n \to \infty$.
\end{lemma}

As long as the event $\mathcal T_R$ holds, we have sufficiently good control on the dependencies between different edges to end up in triangles in $H_R$ and thus getting deleted when passing to $H_R'$. Specifically, for a list of pairs of edges of $K_m$ (corresponding to edge pairs of the form $(\{\pi_R(u_i),\pi_R(u_j)\},\{\pi_R(v_i),\pi_R(v_j)\})$ for $i\neq j$ in some potential odd connected pairing), we are interested in the probability that for at least one edge pair in the list both members are present as edges in $H_R$ and are both not on triangles in $H_R$. If this happens, both members will be present as edges in $H'_R$ and therefore lead to edges in $H_0$. Our next lemma shows that on the event $\mathcal T_R$, we are indeed quite likely to find one such edge pair in the list, with a probability estimate that is strong enough to union-bound over all potential odd connected pairings.

 \begin{lemma}\label{lemma: one connection survives}
    Fix $X\in \{R,B\}$ and $\gamma>0$, and assume that $n$ is sufficiently large with respect to~$\gamma$. Consider a graph $\Gamma$ with vertex set $\binom{[m]}2$ such that $e(\Gamma)\geq \gamma n^2$ and $\Delta_\Gamma\leq 4n^2/m^2$. Then we have
    \[\mathbb P\big[\mathcal T_X\text{ holds and there is no }\{e,f\} \in E(\Gamma)\text{ with }e,f \in E(H'_X)\big] \leq n^{-2n}.\]
\end{lemma}

Note that for $X=R$, the event whose probability is considered in this lemma only depends on the outcomes of $H_R^*$ and $H_R$ (and not on the outcomes of $\pi_R$, $\pi_B$, $H_B^*$ and $H_B$), and similarly for $X=B$ the event only depends on the outcomes of $H_B^*$ and $H_B$.

\section{Proof of the main theorems}\label{sec:main proof}

In this section, we prove \cref{thm:main alpha=2,thm:main} assuming the probabilistic lemmas stated in the previous section. We will need two more combinatorial lemmas, which we prove in the first subsection. The second subsection contains the proof of \cref{thm:odd pairing} (assuming the probabilistic lemmas from \cref{sec:probabilistic lemmas}), from which we then finally deduce \cref{thm:main alpha=2,thm:main} in the last subsection.

\subsection{Combinatorial lemmas}
 Let us say that a pair $\set{u,v}$ of vertices $u,v \in [n]$ is \emph{closed by a red cherry}\margindef{{closed by red or blue cherry}} if there is a vertex $w \in [n]$ so that $\set{u,w},\set{v,w} \in E(H_0)$ and both these edges are red. We define what it means to be \emph{closed by a blue cherry} analogously. The pairs closed by red or blue cherries are precisely those pairs that we would be forced to delete in the creation of $H$ from $H_0$: if a pair $\set{u,v}$ closed by a red or blue cherry is included as an edge of $H_0$, it will not survive into $H$. 
\begin{lemma}\label{lem:closed by cherries}
    Assume that $n$ is sufficiently large. Then, whenever the events $\mathcal D_R$ and $\mathcal F_R$ hold, the number of pairs $\{u,v\} \in \binom{[n]}2$ closed by a red cherry in $H_0$ is at most $n^2/{\log n}$. Similarly, whenever the events $\mathcal D_B$ and $\mathcal F_B$ hold, the number of pairs $\{u,v\} \in \binom{[n]}2$ closed by a blue cherry in $H_0$ is at most $n^2/{\log n}$.
\end{lemma}
\begin{proof}    
	By symmetry it suffices to prove the first statement, regarding red cherries. If $\{u,v\} \in \binom{[n]}2$ is closed by a red cherry, this means that there exists $w$ such that $\set{u,w},\set{v,w}$ are red edges of $H_0$. This, in turn, means that $\set{\pi_R(u),\pi_R(w)}, \set{\pi_R(v),\pi_R(w)}$ are edges of $H_R'$ and hence of $H_R$, so $\pi_R(u)$ and $\pi_R(v)$ must have a common neighbor in the graph $H_R$. Since $\mathcal D_R$ holds, all vertices of $H_R$ have degree at most $2pm$, and so the number of pairs of vertices of $H_R$ with a common neighbor is at most $m\cdot (2pm)^2=4p^2m^3$. Furthermore, for every pair of vertices $x,y\in V_R$ with a common neighbor in $H_R$, there are only $|\pi_R^{-1}(x)|\leq 2n/m$ options for $u\in [n]$ with $\pi_R(u)=x$ and only $|\pi_R^{-1}(y)|\leq 2n/m$ options for $v\in [n]$ with $\pi_R(v)=y$, since the event $\mathcal F_R$ holds. This means that the total number of pairs $\{u,v\} \in \binom{[n]}2$ is closed by a red cherry is at most $4p^2m^3\cdot (2n/m)^2=16p^2mn^2=16n^2/{\log^{4} m}\leq n^2/{\log n}$ (recalling that $m=\lceil n/{\log^{8} n}\rceil$, $p=m^{-1/2}\log^{-2}m$, and $n$ is sufficiently large).
\end{proof}

Recall that our goal is to prove \cref{thm:odd pairing}. So we want to show that the graph $G$ we constructed is unlikely to contain an odd connected pairing of size at least $\eps n$. Recall that any given $s$-pairing $P=\{(u_1,v_1),\dots,(u_{s},v_{s})\}$ in $[n]$ with $s\geq \eps n$ forms an odd connected pairing in $G$ if and only if for all distinct $i,j\in [s]$, at least one of $\{u_i,u_j\}$ and $\{v_i,v_j\}$ is an edge of $G$. To understand whether the pairs $\{u_i,u_j\}$ and $\{v_i,v_j\}$ form edges in $G$, we need to consider their projections $\{\pi_R(u_i),\pi_R(u_j)\}$ and $\{\pi_R(v_i),\pi_R(v_j)\}$ to the vertex set of $H_R$ (and similarly for $H_B$). Our goal is to analyze the probability that for some $i,j\in [s]$ both $\{\pi_R(u_i),\pi_R(u_j)\}$ and $\{\pi_R(v_i),\pi_R(v_j)\}$ form edges of $H_R$, or more precisely, of $H'_R$ (because then $\{u_i,u_j\}$ and $\{v_i,v_j\}$ will be edges of $H_0$). For this analysis, the first important step is to show that a lot of pairs of vertex pairs of $H_R$ appear as $\big\{\{\pi_R(u_i),\pi_R(u_j)\}, \{\pi_R(v_i),\pi_R(v_j)\}\big\}$ for some $i,j\in [s]$. The following lemma shows this, assuming that $\pi_R$ is $\varepsilon$-respectful of $P$ and satisfies $\mathcal F_R$.

\begin{lemma}\label{lem:respectful yields many edges}
	Let $X\in \{R,B\}$, fix $0<\varepsilon<1$, and let $n$ be sufficiently large with respect to $\eps$. Let $P=\{(u_1,v_1),\dots,(u_{s},v_{s})\}$ be an $s$-pairing in $[n]$ such that $s\geq \eps n$. Consider an outcome $\wh{\pi_X}$ of the random map $\pi_X:[n] \to [m]$ which is $\varepsilon$-respectful of $P$ and such that the event $\mathcal F_X$ holds for $\pi_X=\wh{\pi_X}$.
    
    Let $\Pi$ be the set of all pairs $\big\{\{\wh{\pi_X}(u_i),\wh{\pi_X}(u_j)\}, \{\wh{\pi_X}(v_i),\wh{\pi_X}(v_j)\}\big\}$ for all $i,j\in [s]$ with $\wh{\pi_X}(u_i)\ne \wh{\pi_X}(u_j)$ and $\wh{\pi_X}(v_i)\ne \wh{\pi_X}(v_j)$ and also $\{\wh{\pi_X}(u_i),\wh{\pi_X}(u_j)\}\ne \{\wh{\pi_X}(v_i),\wh{\pi_X}(v_j)\}$. Then we have $|\Pi|\geq \eps^4n^2/6$.
\end{lemma}
We can imagine $\Pi$ in \cref{lem:respectful yields many edges} as a graph on the vertex set $\binom{[m]}{2}$. In this graph, for any $i,j\in [s]$, we draw an edge between $\{\wh{\pi_X}(u_i),\wh{\pi_X}(u_j)\}$ and $\{\wh{\pi_X}(v_i),\wh{\pi_X}(v_j)\}$, under the conditions that $\wh{\pi_X}(u_i)\ne \wh{\pi_X}(u_j)$ (as otherwise $\{\wh{\pi_X}(u_i),\wh{\pi_X}(u_j)\}$ does not correspond to an element of $\binom{[m]}{2}$), and similarly $\wh{\pi_X}(v_i)\ne \wh{\pi_X}(v_j)$ as well as $\{\wh{\pi_X}(u_i),\wh{\pi_X}(u_j)\}\ne \{\wh{\pi_X}(v_i),\wh{\pi_X}(v_j)\}$ (since we do not want to create loops in our graph on $\binom{[m]}{2}$). We do not draw parallel edges, so if some pair of vertices in $\binom{[m]}{2}$ appears as $\big\{\{\wh{\pi_X}(u_i),\wh{\pi_X}(u_j)\}, \{\wh{\pi_X}(v_i),\wh{\pi_X}(v_j)\}\big\}$ for multiple different choices of $i,j\in [s]$, we still only draw a single edge. The conclusion of the lemma is that this graph has at least $\eps^4n^2/6$ edges.

\begin{proof}[Proof of \cref{lem:respectful yields many edges}] Without loss of generality let $X=R$ (the case $X=B$ is analogous). Exclusively for this proof, we adopt the notation $\ol v$ to denote $\wh{\pi_R}(v)$. By the assumption that $\wh{\pi_R}:[n] \to [m]$ is $\varepsilon$-respectful of $P$, there exists a subset $I\su [s]$ of size $|I|\geq \eps s\geq \eps^2n$ 
such that $\ol{u_i} \neq \ol{v_i}$ for all $i \in I$, and furthermore the pairs $\{\ol{u_i},\ol{v_i}\}$ for $ i\in I$ are pairwise distinct. 

Note that for any $i,j\in I$ with $\ol{u_i}\ne\ol{u_j}$ and $\ol{v_i}\ne\ol{v_j}$, we also have
$\{\ol{u_i},\ol{u_j}\}\ne \{\ol{v_i},\ol{v_j}\}$.
Indeed, if $\{\ol{u_i},\ol{u_j}\}= \{\ol{v_i},\ol{v_j}\}$, then due to
$\ol{u_i}\ne \ol{v_i}$ and $\ol{u_j}\ne\ol{v_j}$, we would have
$\ol{u_i}=\ol{v_j}$ and $\ol{u_j}=\ol{v_i}$ and hence
$\{\ol{u_i},\ol{v_i}\}=\{\ol{u_j},\ol{v_j}\}$
for the distinct elements $i,j\in I$
(we must have $i\ne j$ because of $\ol{u_i}\ne\ol{u_j}$).
This would be a contradiction.

Recall that $\Pi$ is the set of all pairs $\big\{\{\ol{u_i},\ol{u_j}\}, \{\ol{v_i},\ol{v_j}\}\big\}$ for all $i,j\in [s]$ with $\ol{u_i}\ne \ol{u_j}$ and $\ol{v_i}\ne \ol{v_j}$ and also $\{\ol{u_i},\ol{u_j}\}\ne \{\ol{v_i},\ol{v_j}\}$. Now, let us just consider the pairs
$\{\{\ol{u_i},\ol{u_j}\}, \{\ol{v_i},\ol{v_j}\}\}$
for all $i,j\in I$ with $\ol{u_i}\ne\ol{u_j}$ and $\ol{v_i}\ne\ol{v_j}$ (recall that we then also have $\{\ol{u_i},\ol{u_j}\}\ne \{\ol{v_i},\ol{v_j}\}$). Together, these pairs form some subset $\Pi_I\su \Pi$,
and it suffices to show $|\Pi_I|\geq \eps^4 n^2/6$.

The key claim is that for any given $x,y,z,w\in [m]$, the pair
$\{\{x,y\}, \{z,w\}\}$
can appear as
$\{\{\ol{u_i},\ol{u_j}\}, \{\ol{v_i},\ol{v_j}\}\}$
for at most two different choices of $\{i,j\}\su I$.
Indeed, if
\[
\big\{\{x,y\}, \{z,w\}\big\}
=
\big\{\{\ol{u_i},\ol{u_j}\}, \{\ol{v_i},\ol{v_j}\}\big\},
\]
then the two pairs $\{\ol{u_i},\ol{v_i}\}$ and $\{\ol{u_j},\ol{v_j}\}$
must be either the two pairs $\{x,z\}$ and $\{y,w\}$ or the two pairs
$\{x,w\}$ and $\{y,z\}$.
More precisely,
$\{\{\ol{u_i},\ol{v_i}\}, \{\ol{u_j},\ol{v_j}\}\}$
must be either $\{\{x,z\}, \{y,w\}\}$ or $\{\{x,w\}, \{y,z\}
\}$.
Since the pairs $\{\ol{u_i},\ol{v_i}\}$ are distinct for all $i\in I$,
in either of these two cases there is at most one possibility for
$\{i,j\}\su I$.
Thus, in total, there are indeed at most two choices for $\{i,j\}\su I$.

Consequently, for any $\{i,j\}\su I$ with
$\ol{u_i}\ne\ol{u_j}$ and $\ol{v_i}\ne\ol{v_j}$, the pair
$\{\{\ol{u_i},\ol{u_j}\}, \{\ol{v_i},\ol{v_j}\}\}$
is an element of $\Pi_I$, and every element of $\Pi_I$ is obtained in this
way for at most two choices of $\{i,j\}\su I$.
Therefore we have
\[
|\Pi_I| \geq \frac{ \big| \big\{\{i,j\}\su I : \ol{u_i}\ne\ol{u_j},\, \ol{v_i}\ne\ol{v_j} \big\} \big|
}{2}.
\]
Among the total number $\binom{|I|}{2}$ of pairs $\{i,j\}\su I$,
the number of pairs with $\ol{u_i}=\ol{u_j}$ is at most
$|I|\cdot (2n/m)$.
Indeed, choosing any $j\in I$, in order to have $\ol{u_i}=\ol{u_j}$,
we need to choose $i\in I$ such that $u_i$ is a vertex in the set
$(\wh{\pi_R})^{-1}(x)$ for $x=\ol{u_j}$.
Due to the event $\mathcal F_R$ we have
$|(\wh{\pi_R})^{-1}(x)|\leq 2n/m$, so there are at most $2n/m$ choices for $u_i$
and hence for $i\in I$
(recall that the vertices $u_1,\dots,u_s,v_1,\dots,v_s$ are all distinct).
Similarly, the number of pairs $\{i,j\}\su I$ with
$\ol{v_i}=\ol{v_j}$ is at most $|I|\cdot (2n/m)$.
Thus,
\begin{align*}
\big|
\big\{\{i,j\}\su I :
\ol{u_i}\ne\ol{u_j},\,
\ol{v_i}\ne\ol{v_j}
\big\}
\big|
&\geq
\binom{|I|}{2}
- |I|\cdot \frac{2n}{m}
- |I|\cdot \frac{2n}{m} \\
&=
|I|\left(\frac{|I|-1}{2}-\frac{4n}{m}\right) \\
&\geq
\eps^2 n
\left(\frac{\eps^2 n-1}{2}-\frac{4n}{m}\right)
\geq
\frac{\eps^4 n^2}{3},
\end{align*}
using that $n$ and hence $m$ is sufficiently large with respect to $\eps$.
This implies
$|\Pi|\geq |\Pi_I|\geq \eps^4 n^2/6$, as desired.
\end{proof}

\subsection{Proof of Theorem \ref{thm:odd pairing}}
Given the results of the previous subsection, we are now ready for the proof of \cref{thm:odd pairing}. 

\begin{proof}[Proof of \cref{thm:odd pairing}, assuming the lemmas in \cref{sec:probabilistic lemmas}]
    Note that we may assume without loss of generality that the given fixed value of $\eps$ satisfies $0<\eps<1/10$ (since the statement in \cref{thm:odd pairing} for some value of $\eps$ automatically implies the same statement for all larger values of $\eps$). Let us define $\gamma := \eps^4/10>0$. We are assuming that $n$ is sufficiently large with respect to $\eps$, so $n$ is also large with respect to $\gamma$ (in particular, large enough for the statement in \cref{lemma: one connection survives} to hold).
    
     Recall that the random graph $G$ constructed in \cref{sec:construction} has $n$ vertices and satisfies $\alpha(G)\leq 2$. Our goal is to show that, as $n\to\infty$, with high probability $G$ does not have an odd connected pairing of size at least $\eps n$. This in particular implies that for sufficiently large $n$ (sufficiently large with respect to $\eps$), there exists an $n$-vertex graph $G$ with $\alpha(G)\leq 2$ which does not contain an odd connected pairing of size at least $\eps n$.
    
    Letting $s=\lceil \eps n\rceil$, observe that $G$ has an odd connected pairing of size at least $\eps n$ if and only if it has an odd connected pairing of size $s$, meaning that some $s$-pairing in $[n]$ forms an odd connected pairing in $G$. So we can observe that
    \begin{align*}
    &\pr{G\text{ has an odd connected pairing of size at least }\eps n}\\
    &\leq \pr{\mathcal F_R^\comp}+\pr{\mathcal F_B^\comp}+\pr{\mathcal D_R^\comp}+\pr{\mathcal D_B^\comp}+\pr{\mathcal T_R^\comp}+\pr{\mathcal T_B^\comp}+\pr{\cE(\eps)^\comp}\\
    &\qquad\qquad+\pr{\mathcal F_R, \mathcal F_B, \mathcal D_R, \mathcal D_B, \mathcal T_R, \mathcal T_B, \cE(\eps)\text{ hold and }G\text{ has an odd connected pairing of size }s}\\
    &\leq \pr{\mathcal F_R^\comp}+\pr{\mathcal F_B^\comp}+\pr{\mathcal D_R^\comp}+\pr{\mathcal D_B^\comp}+\pr{\mathcal T_R^\comp}+\pr{\mathcal T_B^\comp}+\pr{\cE(\eps)^\comp}\\
    &\qquad\qquad+\sum_P\pr{\mathcal F_R, \mathcal F_B, \mathcal D_R, \mathcal D_B, \mathcal T_R, \mathcal T_B, \cE(\eps)\text{ hold and }P\text{ is an odd connected pairing in }G},
    \end{align*}
    where the sum is taken over all $s$-pairings $P$ in $[n]$. Note that the number of such $s$-pairings is at most $n^n$. Thus, it suffices to prove
    \begin{equation}\label{eq:to-show-one-P}
        \pr{\mathcal F_R, \mathcal F_B, \mathcal D_R, \mathcal D_B, \mathcal T_R, \mathcal T_B, \cE(\eps)\text{ hold and }P\text{ is an odd connected pairing in }G}\leq n^{-2n}
    \end{equation}
    for every $s$-pairing $P$ in $[n]$. Indeed, then using \cref{lem:easy chernoff,lem:random coupling,lem:easy chernoff triangles}, we can conclude that
    \begin{align*}
    &\pr{G\text{ has an odd connected pairing of size at least }\eps n}\\
    &\qquad\leq o(1)+o(1)+o(1)+o(1)+o(1)+o(1)+o(1)+n^n\cdot n^{-2n}\\
    &\qquad=o(1)+n^{-n}=o(1),
    \end{align*}
    as desired.

    To prove \cref{eq:to-show-one-P} consider an $s$-pairing $P=\{(u_1,v_1),\dots,(u_{s},v_{s})\}$ in $[n]$. It suffices to prove a probability bound as in \cref{eq:to-show-one-P} conditional on any outcomes $\wh{\pi_R}$ and $\wh{\pi_B}$ of the random maps $\pi_R$ and $\pi_B$. If one of the events $\mathcal F_R$,  $\mathcal F_B$ and $\cE(\eps)$ fails for these outcomes $\wh{\pi_R}$ and $\wh{\pi_B}$, the relevant conditional probability is zero. So we may assume that $\mathcal F_R$,  $\mathcal F_B$ and $\cE(\eps)$ hold when $\pi_R=\wh{\pi_R}$ and $\pi_B=\wh{\pi_B}$. Then, by definition of the event $\cE(\eps)$, at least one of the  functions $\wh{\pi_R}$ and $\wh{\pi_B}$ must be $\varepsilon$-respectful of $P$. Let us assume without loss of generality that $\wh{\pi_R}$ is $\varepsilon$-respectful of $P$ (the other case is analogous).

    We can now apply \cref{lem:respectful yields many edges}, since $\wh{\pi_R}:[n]\to [m]$ is $\varepsilon$-respectful of the $s$-pairing $P=\{(u_1,v_1),\dots,(u_{s},v_{s})\}$ (with $s\geq \eps n$), and $\mathcal{F}_R$ holds for $\pi_R=\wh{\pi_R}$. As in the lemma statement, let  $\Pi$ be the set of all pairs $\big\{\{\wh{\pi_R}(u_i),\wh{\pi_R}(u_j)\}, \{\wh{\pi_R}(v_i),\wh{\pi_R}(v_j)\}\big\}$ for all $i,j\in [s]$ with $\wh{\pi_R}(u_i)\ne\wh{\pi_R}(u_j)$ and $\wh{\pi_R}(v_i)\ne\wh{\pi_R}(v_j)$ and also $\{\wh{\pi_R}(u_i),\wh{\pi_R}(u_j)\}\ne \{\wh{\pi_R}(v_i),\wh{\pi_R}(v_j)\}$. Then by \cref{lem:respectful yields many edges} we have $|\Pi|\geq \eps^4n^2/6$.

    Let us now also condition on any fixed outcome $\wh{H_B}$ of the random graph $H_B$. If the event $\mathcal D_B$ fails for $H_B=\wh{H_B}$, the relevant probability we want to bound is again zero. So we may assume that $\mathcal D_B$ holds for $H_B=\wh{H_B}$. Thus, we can apply \cref{lem:closed by cherries}, and conclude that at most $n^2/{\log n}$ pairs $\{u,v\} \in \binom{[n]}2$ are closed by a blue cherry in $H_0$ whenever $H_B=\wh{H_B}$ and $\pi_B=\wh{\pi_B}$. Note that the outcomes of $\pi_B$ and $H_B$ completely determine the blue edges in $H_0$, and they therefore determine which pairs of vertices are closed by a blue cherry in $H_0$.
    
    Let $\Pi_\Gamma\su \Pi$ be the set of all pairs $\big\{\{\wh{\pi_R}(u_i),\wh{\pi_R}(u_j)\}, \{\wh{\pi_R}(v_i),\wh{\pi_R}(v_j)\}\big\}$ for all $i,j\in [s]$ with $\wh{\pi_R}(u_i)\ne\wh{\pi_R}(u_j)$ and $\wh{\pi_R}(v_i)\ne\wh{\pi_R}(v_j)$ and $\{\wh{\pi_R}(u_i),\wh{\pi_R}(u_j)\}\ne \{\wh{\pi_R}(v_i),\wh{\pi_R}(v_j)\}$, such that neither $\{u_i,u_j\}$ nor $\{v_i,v_j\}$ is closed by a blue cherry in $H_0$ when $H_B=\wh{H_B}$ and $\pi_B=\wh{\pi_B}$. Since the number of $\{i,j\}\su [s]$ such that $\{u_i,u_j\}$ or $\{v_i,v_j\}$ is closed by such a blue cherry is at most $n^2/{\log n}$, we find that
    \[|\Pi_\Gamma|\geq |\Pi|-\frac{n^2}{\log n}\geq \frac{\eps^4n^2}{6}-\frac{n^2}{\log n}\geq \gamma n^2,\]
    recalling that $\gamma=\eps^4/10$ and using the assumption that $n$  is sufficiently large with respect to $\eps$. Furthermore, the set $\Pi_\Gamma$ is determined by $\wh{\pi_R}$, $\wh{\pi_B}$ and $\wh{H_B}$ (and $P$).

    Now let $\Gamma$ be the graph with vertex set $\binom{[m]}{2}$ and edge set $\Pi_\Gamma$. Then $e(\Gamma)=|\Pi_\Gamma|\geq \gamma n^2$. We furthermore claim that $\Delta_\Gamma\leq 4n^2/m^2$. Indeed, for any given $\{x,y\}\in \binom{[m]}{2}$, due to $\mathcal F_R$ holding for $\pi_R=\wh{\pi_R}$, we have $|(\wh{\pi_R})^{-1}(x)|\leq 2n/m$ and $|(\wh{\pi_R})^{-1}(y)|\leq 2n/m$, and hence there can be at most $(2n/m)^2=4n^2/m^2$ pairs $\{i,j\}\su [s]$ with $\{\wh{\pi_R}(u_i),\wh{\pi_R}(u_j)\}=\{x,y\}$ or $\set{\wh{\pi_R}(v_i),\wh{\pi_R}(v_j)}=\{x,y\}$. Thus, $\{x,y\}$ appears in at most $4n^2/m^2$ of the pairs $\big\{\{\wh{\pi_R}(u_i),\wh{\pi_R}(u_j)\}, \{\wh{\pi_R}(v_i),\wh{\pi_R}(v_j)\}\big\}$ forming $\Pi$. In particular, each $\{x,y\}\in \binom{[m]}{2}$ appears in at most $4n^2/m^2$ pairs in $\Pi_\Gamma$, and hence $\{x,y\}$ as a vertex of $\Gamma$ has degree at most $4n^2/m^2$. This shows $\Delta_\Gamma\leq 4n^2/m^2$, as claimed.

    Thus, we can apply \cref{lemma: one connection survives} to the graph $\Gamma$ on the vertex set $\binom{[m]}{2}$. We find that
    \begin{equation}\label{eq:conclusion-lemma-44}
    \mathbb P\big[\mathcal T_R\text{ holds and there is no }\{e,f\} \in E(\Gamma)\text{ with }e,f \in E(H'_R)\big] \leq n^{-2n}.
     \end{equation}
    Finally, recall that our goal was to show that, conditional on our fixed outcomes of $\pi_R$, $\pi_B$ and $H_B$ we have (see \cref{eq:to-show-one-P})
    \begin{multline*}
    \mathbb{P}[\mathcal F_R, \mathcal F_B, \mathcal D_R, \mathcal D_B, \mathcal T_R, \mathcal T_B, \cE(\eps)\text{ hold}\\
    \text{and }P\text{ is an odd connected pairing in }G\mid\pi_R=\wh{\pi_R},\pi_B=\wh{\pi_B},H_B=\wh{H_B}]\leq n^{-2n}.
    \end{multline*}
    We claim that whenever there is a pair $\{e,f\} \in E(\Gamma)$ with $e,f \in E(H'_R)$, and $\pi_R=\wh{\pi_R}$, $\pi_B=\wh{\pi_B}$, and $H_B=\wh{H_B}$, the pairing $P$ cannot be an odd connected pairing in $G$. Indeed, as $\{e,f\} \in E(\Gamma)=\Pi_\Gamma$, we must have $\{e,f\}=\big\{\{\pi_R(u_i),\pi_R(u_j)\}, \{\pi_R(v_i),\pi_R(v_j)\}\big\}$ for some indices $i,j\in [s]$ with $\pi_R(u_i)\ne\pi_R(u_j)$ and $\pi_R(v_i)\ne\pi_R(v_j)$ and $\{\pi_R(u_i),\pi_R(u_j)\}\ne \{\pi_R(v_i),\pi_R(v_j)\}$, such that neither $\{u_i,u_j\}$ nor $\{v_i,v_j\}$ is closed by a blue cherry in $H_0$. As $e,f \in E(H'_R)$, we now have $\{\pi_R(u_i),\pi_R(u_j)\}, \{\pi_R(v_i),\pi_R(v_j)\}\in E(H'_R)$. By construction of $H_0$, this means that $\{u_i,u_j\}, \{v_i,v_j\}\in E(H_0)$ are red edges of $H_0$. Since neither $\{u_i,u_j\}$ nor $\{v_i,v_j\}$ is closed by a blue cherry, neither of these red edges gets deleted when passing from $H_0$ to $H$ and so we have $\{u_i,u_j\}, \{v_i,v_j\}\in E(H)$. Recalling that $G$ is the complement of the graph $H$, this means that $\{u_i,u_j\}, \{v_i,v_j\}\not\in E(G)$. Therefore $P$ is not an odd connected pairing in $G$. Thus, we can conclude that
    \begin{align*}
        &\mathbb P[\mathcal F_R, \mathcal F_B, \mathcal D_R, \mathcal D_B, \mathcal T_R, \mathcal T_B, \cE(\eps)\text{ hold}\\
        &\qquad\qquad\text{ and }P\text{ is an odd connected pairing in }G \mid {\pi_R=\wh{\pi_R},\pi_B=\wh{\pi_B},H_B=\wh{H_B}}]\\
        &\leq \cpr{\mathcal T_R\text{ holds and }P\text{ is an odd connected pairing in }G}{\pi_R=\wh{\pi_R},\pi_B=\wh{\pi_B},H_B=\wh{H_B}}\\
        &\leq \cpr{\mathcal T_R\text{ holds and there is no }\{e,f\} \in E(\Gamma)\text{ with }e,f \in E(H'_R)}{\pi_R=\wh{\pi_R},\pi_B=\wh{\pi_B},H_B=\wh{H_B}}\\
        &= \pr{\mathcal T_R\text{ holds and there is no }\{e,f\} \in E(\Gamma)\text{ with }e,f \in E(H'_R)}\\
        &\leq n^{-2n},
    \end{align*}
    where in the penultimate step we used that the events $\mathcal{T}_R$ and $e,f \in E(H'_R)$ for $\{e,f\} \in E(\Gamma)$ only depend on the outcome of the random graphs $H^*_R$ and $H_R$, and are therefore not affected by the conditioning on the outcomes of $\pi_R,\pi_B,H_B$ (which are independent of $H^*_R$ and $H_R$). Furthermore, in the last step we used \cref{eq:conclusion-lemma-44}. This is precisely what we wanted to prove.\end{proof}

\subsection{Proof of Theorems \ref{thm:main} and \ref{thm:main alpha=2}}\label{sec:deduce main thm}
Finally, in order to complete the proofs of \cref{thm:main,thm:main alpha=2}, in this section we give the short deterministic argument showing that \cref{thm:odd pairing} implies \cref{thm:main alpha=2}. At the end of this section, we also include the almost-immediate deduction of \cref{thm:main} from \cref{thm:main alpha=2}.

\begin{proof}[Proof of \cref{thm:main alpha=2}, using \cref{thm:odd pairing}]
	Fix $0<\eta<1$. Our goal is to show that if $n$ is sufficiently large, there is an $n$-vertex graph $G$ with $\alpha(G)\leq 2$ which does not have $K_t$ as an odd minor for $t=\lceil (\frac13+\eta)n\rceil$. Let $\eps= \eta/2$, and let $G$ be an $n$-vertex graph given by \cref{thm:odd pairing}, so that $\alpha(G)\leq 2$ and $G$ does not have an odd connected pairing of size at least $\eps n$. 
    It remains to prove that $G$ does not have $K_t$ as an odd minor.

	So suppose for contradiction that there are vertex-disjoint trees $T_{1},\dots,T_t\subseteq G$ and colorings $\chi_{1},\dots,\chi_t$ as in the definition of an odd minor. Then each of the trees  $T_{1},\dots,T_t$ has at least one vertex. Furthermore $\ab{V(T_1)}+\dots+\ab{V(T_t)}\leq n$, and hence there can be at most $n/3$ indices $i \in [t]$ with $\ab{V(T_i)}\geq 3$. As a consequence, at least $t-(n/3)\geq (\frac13+\eta)n-(n/3)=\eta n$ of the trees $T_{1},\dots,T_t$ have at most two vertices. In particular, letting $s=\lceil \eta n/2\rceil<n/2$, we find that at least $s$ of the trees have exactly two vertices, or at least $s$ of the trees have exactly one vertex.

	Suppose that the first happens, so without loss of generality $T_1,\dots,T_s$ have exactly two vertices each. For every $i \in [s]$, we know that $\chi_i$ is a proper $2$-coloring of the edge $T_i$, hence one vertex of $T_i$ receives color white and the other color black. We can thus define $u_i$ to be the white vertex of $T_i$, and $v_i$ the black vertex. Then $\set{(u_1,v_1),\dots,(u_s,v_s)}$ define an odd connected pairing of size $s$: for all $i\neq j$, we know that there is some monochromatic (with respect to $\chi_i,\chi_j$) edge between $T_i$ and $T_j$, and this edge must be either $\set{u_i, u_j}$ or $\set{v_i, v_j}$.

	In the second case, suppose without loss of generality that $T_1,\dots,T_s$ each consist of a single vertex. Then the fact that these $s$ vertices form an odd $K_s$ minor implies that, for every $i \neq j$, there is an edge between the single vertex of $T_i$ and the single vertex of $T_j$. In particular, we have actually found that $K_s \subseteq G$. Now, label the vertices in $T_1,\dots,T_s$ as $u_1,\dots,u_{s}$, and choose $s$ other arbitrary vertices $v_1,\dots,v_s$, which we may do since $s<n/2$. Then $\{(u_1,v_1),\dots,(u_s,v_s)\}$ is an odd connected pairing since for any distinct $i,j\in [s]$ there is an edge between $u_i$ and $u_j$.

	Thus, in either case, we have found that $G$ contains an odd connected pairing of size $s$. Noting that $s\geq \eta n/2=\eps n$, this is a contradiction to our choice of $G$.
\end{proof}

We end this section by showing that \cref{thm:main} follows from \cref{thm:main alpha=2}.

\begin{proof}[Proof of \cref{thm:main}, using \cref{thm:main alpha=2}]
Let $\delta>0$, and without loss of generality assume $\delta<1$. Set $\eta:=\delta/5>0$. Now, for any sufficiently large $t$, we can apply \cref{thm:main alpha=2} to $n=\lceil (3-10\eta)t\rceil$ and $\eta$, and obtain an $n$-vertex graph $G$ with $\alpha(G)\leq 2$ not having $K_t$ as an odd minor (here, we used that $t\geq (\frac{1}{3}+\eta)(3-9\eta)t\geq (\frac{1}{3}+\eta)n$, if $t$ is sufficiently large). We conclude by observing that $\chi(G)\geq n/\alpha(G)\geq n/2\geq (\frac{3}{2}-5\eta)t=(\frac{3}{2}-\delta)t$.
\end{proof}

\section{Properties of the random functions}\label{sec:coupling}
In this section, we provide the proof of~\cref{lem:random coupling}.
\begin{proof}[Proof of \cref{lem:random coupling}]
		Fix $0<\eps<1/10$, and let $s=\lceil \eps n\rceil$. We need to show that as $n \to \infty$, it holds with high probability that for each $s$-pairing $P$ in $[n]$ at least one of the random functions $\pi_R$ and $\pi_B$ is $\varepsilon$-respectful of $P$. We may in particular assume that $n$ is sufficiently large with respect to $\eps$.
        
        We begin by fixing an $s$-pairing $P = \{(u_1,v_1),\dots(u_{s},v_{s})\}$ in $[n]$, and letting $\pi:[n]\to [m]$ be a uniformly random function. We claim that 
	\begin{equation}\label{eq:one pi}
		\pr{\pi \text{ is not }\varepsilon\text{-respectful of }P} \leq n^{-(5/8)\eps n}.
	\end{equation}
	First, note that if $\pi$ is not $\varepsilon$-respectful of $P$, then this means that there is a set $T \subseteq \binom{[m]}2$ of size $\ab T =  \lfloor \varepsilon s\rfloor$ such that for all $i \in [s]$ we have $\{\pi(u_i),\pi(v_i) \} \in T$ or $\pi(u_i)=\pi(v_i)$.

	For a fixed set $T \subseteq \binom{[m]}2$ of size $\ab T  =  \lfloor \varepsilon s\rfloor$ and a fixed index $i \in [s]$, we have
    \[\pr{\{\pi(u_i),\pi(v_i)\} \in T \text{ or }\pi(u_i)=\pi(v_i)} \leq \pr{\pi(u_i)=\pi(v_i)} + \pr{\{\pi(u_i),\pi(v_i)\}\in T}.\]
    The first summand on the right-hand side is exactly $1/m$, and the second summand is exactly $2|T|/m^2$. Indeed, $(\pi(u_i),\pi(v_i))\in [m]^2$ is uniformly distributed among all $m^2$ options, and precisely $2|T|$ of these options correspond to pairs in $T \subseteq \binom{[m]}2$ (note that for every pair in $T$ there are two possible orderings, meaning that there are two corresponding elements of $[m]^2$). In total, we find that
    \[
        \pr{\{\pi(u_i),\pi(v_i)\} \in T \text{ or }\pi(u_i)=\pi(v_i)} \leq \frac{1}{m}+\frac{2|T|}{m^2}=\frac{2|T|+m}{m^2}\leq \frac{2\eps^2 n+2+m}{m^2}\leq \frac{3\eps^2 n}{m^2}\leq n^{-7/8}
    \]
    for every $i \in [s]$ (here, we used that $\ab T  =  \lfloor \varepsilon s\rfloor\leq \eps(\eps n+1)< \eps^2 n+1$ and $m=\lceil n/(\log n)^{8}\rceil$, and that $n$ is sufficiently large with respect to $\eps$). Now, the outcomes of $\pi(u_i)$ and $\pi(v_i)$ are independent over all the different $i \in [s]$. Therefore, we obtain
	\[
		\pr[\big]{\text{for all }i \in [s]\text{ we have }\{\pi(u_i),\pi(v_i)\} \in T \text{ or }\pi(u_i)=\pi(v_i)} \leq (n^{-7/8})^{s}\leq n^{-(7/8)\eps n}.
	\]
	Applying the union bound over the
    \[\binom{\binom m2}{\lfloor \varepsilon s\rfloor}\leq \binom m2^{\varepsilon s}\leq (n^2)^{\eps^2n+1}= n^{2\eps^2n+2}\leq n^{\eps n/4}\]
    choices of $T\subseteq \binom{[m]}2$ of size $\ab T =  \lfloor \varepsilon s\rfloor$ (where we again used that $n$ is sufficiently large with respect to $\eps$), we find that
    \[\pr{\pi \text{ is not }\varepsilon\text{-respectful of }P} \leq  n^{\eps n/4}\cdot n^{-(7/8)\eps n}= n^{-(5/8)\eps n},\]
	proving \eqref{eq:one pi}. As a consequence of \eqref{eq:one pi}, we see that for the fixed pairing $P$, we have
    \[\pr{\pi_R \text{ and }\pi_B\text{ are both not }\varepsilon\text{-respectful of }P} = \big(\pr{\pi \text{ is not }\varepsilon\text{-respectful of }P}\big)^2\leq n^{-(5/4)\eps n},\]
	since $\pi_R,\pi_B$ are independent and both uniformly random among all maps $[n]\to [m]$. Finally, 
	we note that the number of $s$-pairings $P$ in $[n]$ is 
	at most $\binom n{s}^2 s!$, as we may pick the set $\{u_1,\dots,u_{s}\}$, then pick the set $\{v_1,\dots,v_{s}\}$, and finally pick a bijection between them encoding the pairs $\{u_i,v_i\}$. This number is
	\[
		\binom n{s}^2 s! \leq 2^{2n} s^{s} \leq 2^{2n} n^{s} \leq 2^{2n} n^{\eps n+1}=n^{\eps n+1+2n/{\log_2 n}}\leq n^{(9/8)\eps n}.
	\]
	Applying a union bound over all choices of $P$, we find that
    \begin{align*}
    &\pr{\text{for some $s$-pairing }P\text{ in }[n]\text{ neither }\pi_R\text{ nor }\pi_B\text{ is $\eps$-respectful of }P}\\
    &\qquad\qquad\qquad\qquad\qquad\qquad\qquad\qquad\qquad\qquad\qquad\qquad\leq n^{(9/8)\eps n}\cdot n^{-(5/4)\eps n}=n^{-\eps n/8}.
    \end{align*}
    Thus, for $n$ sufficiently large with respect to $\eps$, the probability that the event $\cE(\eps)$ fails is at most $n^{-\eps n/8}$. In particular, $\cE(\eps)$ holds with high probability as $n\to \infty$.
\end{proof}

\section{Properties of the random graphs}\label{sec:connection survives}
Our main goal in this section is to prove \cref{lemma: one connection survives}; in \cref{subsec:concentration ineqs} we also prove the simple \cref{lem:easy chernoff triangles,lem:easy chernoff}.

\subsection{A preparatory lemma}
In order to prove \cref{lemma: one connection survives}, we will need \cref{lem:first-two-sparsifications} below. Before stating and proving this lemma,
we state a concentration inequality, closely related to Talagrand's inequality \cite{MR1361756}, which was proved by McDiarmid and Reed~\cite[Theorem~2]{MR2268235}. 

\begin{lemma}\label{lemma: talagrand}
Let $r>0$, let $X_1,\ldots,X_N$ be independent random variables taking values in a finite set $S$, and let $X=f(X_1,\ldots,X_N)$ for some $f\colon S^N\to\bR_{\ge 0}$. Suppose that for all~$i\in [N]$ and $\mathbf{x}\in S^N$, changing the~$i$-th coordinate of~$\mathbf{x}$ changes the value of~$f(\mathbf{x})$ by at most~$1$. Furthermore, suppose that for every $\mathbf{x}\in S^N$ there exists a subset $J\subseteq [N]$ with $|J|\le rf(\mathbf{x})$ such that every $\mathbf{y}\in S^N$ with $\mathbf{y}_i=\mathbf{x}_i$ for all $i\in J$ satisfies $f(\mathbf{y})\ge f(\mathbf{x})$. Then, for all~$t\ge 0$, we have
  \begin{equation*}
    \pr[\big]{X\leq \ex{X}-t}\leq \exp\paren[\bigg]{-\frac{t^2}{2\cdot (r\ex{X}+t/3)}}.
  \end{equation*}
\end{lemma}

The key lemma in the proof of \cref{lemma: one connection survives} is the following lemma about a random vertex subset $V^*$ of a given graph $\Gamma$, containing each vertex of $\Gamma$ independently with some probability $q$. The lemma states that, with large probability, in $\Gamma[V^*]$ we can find a large matching (i.e.\ a large set of pairwise disjoint edges). We will apply this lemma with $\Gamma$ being the graph on $\binom{[m]}2$ given in \cref{lemma: one connection survives}, and the random subset $V^*\su \binom{[m]}2$ corresponding to the edge set of the random graph $H^*_X$.

\begin{lemma}\label{lem:first-two-sparsifications}
 Let~$\Gamma$ be a graph, and consider a random vertex set~$V^*\subseteq V(\Gamma)$ where every vertex of~$\Gamma$ is included independently with some probability~$q\leq 1/\Delta_\Gamma$.  Then with probability at least
  \[1-\exp\paren[\bigg]{-\frac{q^2e\paren{\Gamma}}{200}},\] there exists a 
  matching $M^\ast$ of size at least $q^2e\paren{\Gamma}/20$ in the induced subgraph $\Gamma[V^*]$. 
\end{lemma}

To prove the lemma, we apply \cref{lemma: talagrand} to the matching number of $\Gamma[V^*]$. 
\begin{proof}[Proof of \cref{lem:first-two-sparsifications}]
If $e(\Gamma)=0$, the statement in the lemma trivially holds, so we may assume $e(\Gamma)\geq 1$, and hence $\Delta_\Gamma\geq 1$.

We need to show that with very large probability the graph $\Gamma[V^*]$ contains a matching of size at least $q^2e\paren{\Gamma}/20$, i.e., it has matching number $\nu(\Gamma[V^*])\geq q^2e\paren{\Gamma}/20$. To do so, we first show a lower bound on $\mathbb{E}[\nu(\Gamma[V^*])]$. 
       Let us consider the set $E^*$ consisting of all edges $\{u,v\}\in E(\Gamma[V^*])$  with $d_{\Gamma[V^*]}(u)=d_{\Gamma[V^*]}(v)=1$. In other words, $E^*$ is the collection of all edges $\{u,v\}\in E(\Gamma)$ with $u,v\in V^*$ where both $u$ and $v$ have degree $1$ in $\Gamma[V^*]$. Note that then the edges $e\in E^*$ are automatically pairwise disjoint (because if two such edges were to meet in a vertex, then this vertex would need to have degree at least $2$), and thus we always have $\nu(\Gamma[V^*])\ge |E^*|$.  
       Observe that for any fixed edge $\{u,v\}\in E(\Gamma)$ we have
    \begin{align*}
    \pr{\{u,v\}\in E^*}&= \pr{u,v\in V^*\text{ and }d_{\Gamma[V^*]}(u)= d_{\Gamma[V^*]}(v)= 1}\\
    &\geq \pr{u,v\in V^*\text{ and }x\not\in V^*\text{ for all }x\in N_\Gamma(u)\cup N_\Gamma(v)\setminus\{u,v\}}\\
    &\geq q^2\cdot (1-q)^{2(\Delta_\Gamma-1)}\geq q^2\cdot \Big(1-\frac{1}{\Delta_\Gamma}\Big)^{2(\Delta_\Gamma-1)} \geq q^2\cdot e^{-2}\geq \frac{q^2}{10},
    \end{align*}
    where we used the assumption $q\le 1/\Delta_\Gamma$ and that $(1-1/t)^{t-1}\geq 1/e$ holds for all positive integers $t$. Thus, we obtain
    \[\ex{\nu(\Gamma[V^*])}\ge \ex{\abs{E^*}}\geq e\paren{\Gamma}\cdot \frac{q^2}{10}=\frac{q^2e\paren{\Gamma}}{10}.\]
     Note that we can view $\nu(\Gamma[V^*])$ as a function of the  indicator random variables for the events $v\in V^*$ for $v\in V(\Gamma)$. Further note that changing only one of these variables can change $\nu(\Gamma[V^*])$ by at most $1$ (since adding or omitting a vertex from $V^*$ changes the matching number of $\Gamma[V^*]$ by at most $1$). We claim that, taking $r=2$, also the second condition in \cref{lemma: talagrand} is satisfied. Indeed, for every vertex subset $X\subseteq V(\Gamma)$, we can take $J\su X\su V(\Gamma)$ to be the vertex set of a matching of size $\nu(\Gamma[X])$ in $\Gamma[X]$. Then we have $|J|= 2\nu(\Gamma[X])$, and we furthermore have $\nu(\Gamma[Y])\ge \nu(\Gamma[X])$ for every $Y\subseteq V(\Gamma)$ with $J\subseteq Y$. 
     Hence, we can apply \cref{lemma: talagrand} to the random variable $\nu(\Gamma[V^*])$ with the parameters $r=2$ and $t=\ex{\nu(\Gamma[V^*])}/2$ to obtain:
         \begin{align*}
    \pr[\bigg]{\nu(\Gamma[V^*])< \frac{q^2e\paren{\Gamma}}{20}}&\leq \pr[\bigg]{\nu(\Gamma[V^*])< \ex{\nu(\Gamma[V^*])}-\frac{\ex{\nu(\Gamma[V^*])}}{2}}\notag\\
    &\leq \exp\paren[\bigg]{-\frac{(\ex{\nu(\Gamma[V^*])}/2)^2}{2\cdot \left(2\ex{\nu(\Gamma[V^*])}+\ex{\nu(\Gamma[V^*])}/6\right)}}\notag\\
    &\le\exp\paren[\bigg]{-\frac{\ex{\nu(\Gamma[V^*])}}{20}}\leq \exp\paren[\bigg]{-\frac{q^2e\paren{\Gamma}}{200}}.
    \qedhere
    \end{align*} 
\end{proof}

\subsection{Proof of Lemma~\ref{lemma: one connection survives}} We are now ready to prove \cref{lemma: one connection survives}. In the proof, we apply \cref{lem:first-two-sparsifications} to the given graph $\Gamma$ with vertex set $\binom{[m]}2$ and the random graph $H_X^*$, whose edges can be viewed as a random subset of the vertex set $V(\Gamma)$. We can then conclude that with very large probability there exists a large matching~$M^*$ in $\Gamma$ such that $e,f\in E(H^*_X)$ for all $\{e,f\}\in M$. Therefore, we may condition on a fixed outcome of $H^*_X$ such that this holds, and such that the event $\mathcal{T}_X$ holds (meaning that every edge in $H^*_X$ is contained in only few triangles). We can then select a submatching $M\su M^*$ such that not just all $\{e,f\}\in M^*$ are disjoint from one another, but also the sets of those edges that together with $e$ or $f$ lie on a triangle in $H^*_X$ are disjoint for all $\{e,f\}\in M$. For such a submatching $M\su M^*$, the events of having $e,f\in E(H'_X)$ (i.e., having $e,f\in E(H_X)$ and neither $e$ nor $f$ lying on a triangle in $H_X$) are independent for all $\{e,f\}\in M$. This allows us to conclude that at least one of these events is likely to happen.

\begin{proof}[Proof of Lemma~\ref{lemma: one connection survives}]
Without loss of generality, let $X=R$ (the case $X=B$ is analogous). We are given a graph $\Gamma$ with vertex set $\binom{[m]}2$ such that $e\paren{\Gamma}\geq \gamma n^2$ and $\Delta_\Gamma \leq 4n^2/m^2 \leq 4\log^{16} n \leq 8 \log^{16}m$. The edge set of the random graph $E(H^*_R)$ can be interpreted as a random subset of $\binom{[m]}2$ which includes every element of $\binom{[m]}2$ independently with probability $m^{-1/2}$. Thus, $E(H^*_R)$ is a random subset of the vertex set of $\Gamma$, and we can apply \cref{lem:first-two-sparsifications}, noting that if $n$ (and thus $m$) is sufficiently large, we have $m^{-1/2}\leq 1/(8 \log^{16}m)\leq 1/\Delta_\Gamma$. Let $\cH$ denote the event (depending only on the outcome of $H^*_R$) that there exists a matching~$M^*$ of size $|M^*|\geq (m^{-1/2})^2e\paren{\Gamma}/20=e(\Gamma)/(20m)$
in the induced subgraph $\Gamma[E(H^*_R)]$, i.e.\ a matching $M^*$ of size $|M^*|\geq e(\Gamma)/(20m)$
in $\Gamma$ such that $e,f\in E(H^*_R)$ for all $\{e,f\}\in M^*$. By \cref{lem:first-two-sparsifications} we have
\begin{equation*}
    \pr{\cH}\geq 1-\exp\paren[\bigg]{-\frac{(m^{-1/2})^2e(\Gamma)}{200}}
    \geq 1-\exp\paren[\bigg]{-\frac{\gamma n^2}{200m}}\geq 1-\exp\paren[\big]{-n\log^{7}n},
\end{equation*}
recalling that $m=\lceil n/{\log^{8}n}\rceil$ and $n$ is sufficiently large with respect to $\gamma$. 

It now suffices to show that
\begin{equation}\label{eq:to-show-for-Gamma-lemma-new}
\cpr[\big]{\text{there is no }\{e,f\} \in E(\Gamma)\text{ with }e,f \in E(H'_R)}{\cH\text{ and }\mathcal T_R}\leq \exp\paren[\big]{-3n\log n}.
\end{equation}
Indeed, upon proving this we can conclude that 
\begin{align*}
    &\mathbb P\big[\mathcal T_R\text{ holds and there is no }\{e,f\} \in E(\Gamma)\text{ with }e,f \in E(H'_R)\big]\\
    &\qquad\leq \mathbb P\big[\cH\text{ and }\mathcal T_R\text{ hold and there is no }\{e,f\} \in E(\Gamma)\text{ with }e,f \in E(H'_R)\big]+\mathbb P\big[\cH^\comp\big]\\
    &\qquad\leq \cpr[\big]{\text{there is no }\{e,f\} \in E(\Gamma)\text{ with }e,f \in E(H'_R)}{\cH\text{ and }\mathcal T_R}+\mathbb P\big[\cH^\comp\big]\\
    &\qquad\leq \exp(-3n\log n)+ \exp(-n\log^7 n)\leq \exp(-2n\log n)=n^{-2n},
\end{align*}
as desired. So it only remains to prove \cref{eq:to-show-for-Gamma-lemma-new}.

Recall that both of the events $\cH$ and $\mathcal T_R$ are determined by the outcome of the random graph~$H^*_R$. Thus, to show \cref{eq:to-show-for-Gamma-lemma-new}, we can condition on a fixed outcome $\wh{H^*_R}$ of $H^*_R$ satisfying $\cH$ and $\mathcal T_R$. This means that there is a matching~$M^*$ in $\Gamma$ of size $|M^*|\geq e\paren{\Gamma}/(20m)\geq \gamma n^2/(20m)$ with $e,f\in E(\wh{H^*_R})$ for all $\{e,f\}\in M^*$, and furthermore that every edge in $\wh{H^*_R}$ lies on at most $20\log m$ triangles of $\wh{H^*_R}$. In particular, for any edge $e\in E(\wh{H^*_R})$ there are at most $40\log m$ other edges appearing on a triangle in $\wh{H^*_R}$ together with $e$. For any $\{e,f\}\in M^*$, let $S_{\{e,f\}}\su E(\wh{H^*_R})$ be the set of edges that are equal to $e$ or $f$ or appear on a triangle in $\wh{H^*_R}$ together with $e$ or $f$. Then we have $|S_{\{e,f\}}|\leq 40\log m+40\log m+2\leq 100\log m$ for any $\{e,f\}\in M^*$, and furthermore every edge appears in $S_{\{e,f\}}$ for at most $40\log m+1\leq 50\log m$ different $\{e,f\}\in M^*$.

    Now, for any $\{e,f\}\in M^*$ there are at most $100\log m\cdot 50\log m= 5000\log^2m$ different $\{e',f'\}\in M^*$ with $S_{\{e,f\}}\cap S_{\{e',f'\}}\neq \varnothing$. Therefore, via a simple greedy procedure (or alternatively by the bound of the independence number of a graph in terms of its maximum degree), we can choose a submatching $M\su M^*$ of size
\[|M|\geq \frac{|M^*|}{5000\log^2m+1}\geq \frac{\gamma n^2/(20m)}{10^4\log^2m}\geq\frac{\gamma n^2}{10^6m\log ^2m}\geq \frac{\gamma}{2\cdot 10^6}\cdot n\log^6 n\geq 6n\log^5 n\]
such that for any distinct $\{e,f\},\{e',f'\}\in M$ we have $S_{\{e,f\}}\cap S_{\{e',f'\}}=\varnothing$.

Now, conditional on the event $H^*_R=\wh{H^*_R}$, the random graph $H_R$ is obtained as a random subgraph $H_R\su \wh{H^*_R}$ that includes every edge of $\wh{H^*_R}$ independently with probability $1/{\log^{2} m}$. Note that for every $\{e,f\}\in M\su M^*\su E(\Gamma)$, whenever we have $S_{\{e,f\}}\cap E(H_R)=\{e,f\}$, we have $e,f \in E(H_R)$ and no triangle in $H_R$ contains $e$ or $f$. Now, the events $S_{\{e,f\}}\cap E(H_R)=\{e,f\}$ are independent for all $\{e,f\}\in M$, and for each $\{e,f\}\in M$ we have
\begin{align*}
\cpr[\big]{S_{\{e,f\}}\cap E(H_R)=\{e,f\}}{H_R^*=\wh{H^*_R}}&\geq \Big(\frac{1}{\log^{2} m}\Big)^2\cdot \Big(1-\frac{1}{\log^{2} m}\Big)^{|S_{\{e,f\}}|-2}\\
&\geq \frac{1}{\log^{4} m}\cdot \Big(1-\frac{1}{\log^{2} m}\Big)^{100\log m}\geq \frac{1}{2\log^{4} m}.
\end{align*}
Thus, we obtain
\begin{align*} 
&\cpr[\big]{\text{there is no }\{e,f\} \in M\text{ with\ }e,f \in E(H_R)\text{ s.t.\ no triangle in }H_R\text{ contains }e\text{ or }f}{H_R^*=\wh{H^*_R}}\\
&\quad\leq \cpr[\big]{\text{there is no }\{e,f\} \in M\text{ with }S_{\{e,f\}}\cap E(H_R)=\{e,f\}}{H_R^*=\wh{H^*_R}}\\
&\quad\leq \bigg(1-\frac{1}{2\log^{4} m}\bigg)^{|M|}\leq \exp\paren[\bigg]{-\frac{|M|}{2\log^{4} m}}\leq \exp\paren[\bigg]{-\frac{6n\log^5 n}{2\log^{4} m}}\le \exp\paren[\big]{-3n\log n}.
\end{align*}
For any $\{e,f\} \in M\su M^*\su E(\Gamma)$ with $e,f \in E(H_R)$ such that no triangle in $H_R$ contains $e$ or $f$, by the definition of $H'_R$ we have $e,f \in E(H'_R)$. Thus, we can conclude
\[\cpr[\big]{\text{there is no }\{e,f\} \in E(\Gamma)\text{ with }e,f \in E(H'_R)}{H^*_R=\wh{H^*_R}}\leq \exp\paren[\big]{-3n\log n},\]
as desired.
\end{proof}

\subsection{Proof of Lemmas \ref{lem:easy chernoff} and \ref{lem:easy chernoff triangles}}\label{subsec:concentration ineqs}

In this subsection we derive \cref{lem:easy chernoff,lem:easy chernoff triangles} from the Chernoff bound.
We will use the following simple form of the Chernoff bound, which follows for example from \cite[Theorem~2.8 and eq.~(2.10)]{MR1782847}.

\begin{lemma}\label{lemma: chernoff}
  Suppose that~$X$ is a sum of independent indicator random variables.
  Then, for all~$t\geq 2\cdot \ex{X}$, we have
  \begin{equation*}
    \pr{X\geq t}\leq \exp\paren[\bigg]{-\frac{t}{6}}.
  \end{equation*}
\end{lemma}

With this tool, the proofs of \cref{lem:easy chernoff,lem:easy chernoff triangles} are quite straightforward.

\begin{proof}[Proof of \cref{lem:easy chernoff}]
    It suffices to prove the results for $\mathcal D_R$ and $\mathcal F_R$, since the proofs for $\mathcal D_B$ and $\mathcal F_B$ are analogous. First, fix some vertex $v \in V_R$, and let $X_v$ be its degree in $H_R$. Then $X_v$ is distributed as a binomial random variable with parameters $(m-1,p)$, so in particular $\ex{X_v}\leq pm =(2pm)/2$. Thus, \cref{lemma: chernoff} implies $\pr{X_v \geq 2pm}\leq \exp(-2pm/6) = \exp(-m^{1/2}\log^{-2} m/3)\leq m^{-2}$, where the final inequality holds for $n$ and hence $m$ sufficiently large. Thus, taking a union bound over all $v \in V_R$ proves that $\mathcal D_R$ holds with high probability. 

    Similarly, for fixed $v \in V_R$, let $Y_v$ denote the random variable $\ab{\pi_R^{-1}(v)}$. We thus have that $Y_v$ is distributed as a binomial random variable with parameters $(n,1/m)$, so $\ex{Y_v}=n/m$. We may thus apply \cref{lemma: chernoff} to see that $\pr{Y_v\geq 2n/m} \leq \exp(-(2n/m)/6) \leq \exp(-\log^{8} n/4) \leq n^{-2}\leq m^{-2}$, where the penultimate inequality again holds for sufficiently large $n$. Another union bound over all $v \in V_R$ shows that $\mathcal F_R$ holds with high probability.
\end{proof}

\begin{proof}[Proof of \cref{lem:easy chernoff triangles}]
  It suffices to show that $\mathcal T_R$ holds with high probability; the proof for $\mathcal T_B$ is analogous. For distinct~$u,v\in V(H^*_R)$, the number~$X_{u,v}$ of common neighbors of $u$ and $v$ in $H^*_R$ is a binomial random variable with parameters $(m-2,m^{-1})$. Indeed, any vertex $w\in V(H^*_R)\setminus \{u,v\}$ has edges to both $u$ and $v$ with probability $(m^{-1/2})^2=m^{-1}$, and this happens independently for all $w\in V(H^*_R)\setminus \{u,v\}$. Since~$\ex{X_{u,v}}\leq 1\leq (20 \log m)/2$, Lemma~\ref{lemma: chernoff} yields~$\pr{X_{u,v}\geq 20\log m}\leq\exp(-3\log m)=m^{-3}$. Taking a union bound over the at most $m^2$ choices for $u,v\in V(H^*_R)$ shows that with high probability we have $X_{u,v}\leq 20\log m$ for all distinct $u,v\in V(H^*_R)$. Whenever this happens, every edge in $H^*_R$ is contained in at most $20\log m$ triangles, and so $\mathcal T_R$ holds.
\end{proof}

\section{Concluding remarks}\label{sec:conclusion}

Given that we have disproved the odd Hadwiger conjecture, and given the close connection between Hadwiger's conjecture and its odd variant, it is natural to wonder whether our approach could disprove Hadwiger's conjecture itself. At the moment, there appear to be serious obstructions to achieving this.

The most natural thing one could try in order to disprove Hadwiger's conjecture is to simply analyze the presence of $K_{n/2}$ as a minor in our construction, and aim to prove that with high probability, the constructed graph does not have $K_{n/2}$ as a minor. There are at least two serious issues with making such an approach work. The first concerns the numbers: if we now consider $K_{n/2}$ as an ordinary minor (no longer required to be odd), then we have four potential connections available between two given trees $T_i$ and $T_j$ of order $2$. We would thus want $p^4 n^2 \gg n \log n$ for our union bound to work, requiring $p \gg n^{-1/4}$. Since we also need $p \leq c/\sqrt m$ in order to delete all triangles from $G(m,p)$ without removing many edges, this would require $m \ll \sqrt n$; however, several parts of our proof completely break down at such a small value of $m$.

The other issue seems perhaps even more serious. Recall that in our proof, we immediately dropped one of the conditions defining an odd connected matching, and instead passed to working with an odd connected pairing; namely, we did not require that our branch sets $T_i$ actually define a connected subgraph of $G$. For odd minors, dropping this condition turned out to be immaterial. However, for ordinary minors, one cannot simply drop this condition: it is easy to see that in any proper coloring of $G$ with $\chi(G)$ colors, there must be at least one edge between every two color classes. Thus, if we allow branch sets that are not connected, then every graph $G$ contains such a ``weak minor'' of order $\chi(G)$, simply obtained by contracting each color class of any optimal coloring. Thus, any argument that aims to disprove Hadwiger's conjecture must crucially use that the branch sets are connected, and it is difficult for us to see how this extra condition could significantly affect the relevant probabilities. 

\subsection*{Acknowledgments:} We would like to thank Micha Christoph, Jacob Fox, Marc Kaufmann, Anders Martinsson, and Rob Morris for helpful and inspiring discussions related to this topic. 

\bibliographystyle{yuval}
\bibliography{references}
\appendix
\crefalias{section}{appendix}

\section{Optimality of the factor \texorpdfstring{$\frac 32$}{3/2}}\label{appendix:optimality}
In this section, we prove Proposition~\ref{prop:stupid}, whose proof is inspired by similar arguments used in~\cite{ji2025oddcliqueminorsgraphs}. In the proof, we make use of the following result of Gallai~\cite{MR188100}. Recall that a graph is called \emph{$k$-critical} if it has chromatic number $k$ but every proper subgraph is $(k-1)$-colorable.
\begin{theorem}\label{thm:gallai}
    Let $k\geq 2$ be an integer, and let $G$ be a $k$-critical graph. If $|V(G)|\leq 2k-2$, then there exists a partition $V(G)=X_1\cup X_2$ of $V(G)$ into two non-empty sets $X_1$ and $X_2$ such $\{x_1,x_2\}\in E(G)$ for all $x_1\in X_1$ and $x_2\in X_2$.
\end{theorem}

\noindent With this tool, we are ready to present the proof of \cref{prop:stupid}.

\begin{proof}[Proof of Proposition~\ref{prop:stupid}]
We prove the statement in the proposition by contradiction. Supposing the  statement is false, let $t$ and $G$ be chosen to form a counterexample minimizing $n:=|V(G)|$. Then $\alpha(G)\leq 2$, the graph $G$ does not contain $K_t$ as an odd minor, and we have $k:=\chi(G)>\lceil\frac{3}{2}(t-1)\rceil$ and therefore $k-1\geq \lceil\frac{3}{2}(t-1)\rceil$. By~\cite[Theorem~1.8]{kawasong}, $G$ contains $K_{\lceil n/3\rceil}$ as an odd minor, so we have $\lceil n/3\rceil\leq t-1$ and hence $n\leq 3t-3\leq 2k-2$. Observe that by our minimality assumption on $|V(G)|$, we have $\chi(G-v)=k-1$ for every $v\in V(G)$. Letting $G'$ be a $k$-critical subgraph of $G$, we must therefore have $V(G')=V(G)$, and by Theorem~\ref{thm:gallai} we find that there exists a partition $V(G)=V(G')=X_1\cup X_2$ of $V(G')=V(G)$ into non-empty subsets $X_1$ and $X_2$ such that all possible edges between $X_1$ and $X_2$ exist in $G'$ and thus also in $G$. 

Let $G_1:=G[X_1]$ and $G_2:=G[X_2]$, and note that we have $\chi(G)=\chi(G_1)+\chi(G_2)$. For $i\in \{1,2\}$, let $t_i$ denote the largest positive integer such that $G_i$ has $K_{t_i}$ as an odd minor. By our minimality assumption on $|V(G)|$, we then have $\chi(G_i)\leq \lceil\frac{3}{2}t_i \rceil$ for $i\in \{1,2\}$.

Finally, for $i\in \{1,2\}$, let $Y_i\subseteq X_i$ be chosen as an inclusion-wise minimal subset of $X_i$ such that $G[Y_i]$ still contains $K_{t_i}$ as an odd minor, and let $T_1^i,\ldots,T_{t_i}^i$ be a collection of vertex-disjoint trees in $G[Y_i]$ as in the definition of odd minors. This means that for $i\in \{1,2\}$ there exist proper colorings of the vertices of each of the trees $T_{1}^i,\ldots,T_{t_i}^i$ with the colors black and white, such that any two of these trees are connected by some monochromatic edge. Note that each of these trees  is either a single vertex or has both black and white vertices, and that for each $i\in \{1,2\}$ all single-vertex trees among $T_1^i,\ldots,T_{t_i}^i$ must have the same color. Swapping the colors black and white if needed in the colorings of $T_{1}^i,\ldots,T_{t_i}^i$ for one or both $i\in \{1,2\}$, we may assume without loss of generality that for each $i\in \{1,2\}$, all trees $T_1^i,\ldots,T_{t_i}^i$ contain a white vertex.

Now, by considering the collection of trees $T_1^1,\ldots,T_{t_1}^1, T_1^2,\ldots,T_{t_2}^2$ we find $K_{t_1+t_2}$ as an odd minor in $G$. This means that $t_1+t_2\leq t-1$. Moving on, we distinguish two cases.

\noindent \textbf{Case 1: There exist $i\in \{1,2\}$ and $y\in Y_i$ such that $(X_i\setminus Y_i)\cup \{y\}$ is independent.}

Let us assume without loss of generality that $i=1$ (the case $i=2$ is analogous). In this case, we can obtain a better estimate on the chromatic number of $G_1$, as follows. By minimality of $Y_1$, the graph $G_1':=G_1-((X_1\setminus Y_1)\cup \{y\})=G[Y_1\setminus \{y\}]$ does not contain $K_{t_1}$ as an odd minor. Thus, since $\alpha(G_1')\leq \alpha(G)\leq 2$, and by minimality of $G$, we obtain $\chi(G_1)\leq \chi(G_1')+1\leq  \left\lceil\frac{3}{2}(t_1-1)\right\rceil +1$. We therefore have
\[\chi(G)=\chi(G_1)+\chi(G_{2})\leq \left(\left\lceil\frac{3(t_1-1)}{2}\right\rceil +1\right)+\left\lceil\frac{3t_{2}}{2}\right\rceil=\left\lceil\frac{3(t_1-1)}{2}\right\rceil +\left\lceil\frac{3t_{2}}{2}\right\rceil+1.\]
Consider first the case that $t_1$ is even and $t_2$ is odd. Then we have $\big\lceil\frac{3(t_1-1)}{2}\big\rceil=\frac{3(t_1-1)}{2}+\frac12$ and $\big\lceil\frac{3t_2}{2}\big\rceil=\frac{3t_2}{2}+\frac12$, as well as $\big\lceil\frac{3(t_1+t_2)}{2}\big\rceil=\frac{3(t_1+t_2)}{2}+\frac12$ (since $t_1+t_2$ is odd). So we can conclude
\[\chi(G)\leq \left(\frac{3(t_1-1)}{2} +\frac12\right)+\left(\frac{3t_2}{2}+\frac12\right)+1=\frac{3(t_1+t_2)}{2} +\frac12=\left\lceil\frac{3(t_1+t_2)}{2}\right\rceil\leq \left\lceil\frac{3(t-1)}{2}\right\rceil.\]
In all other cases, we have that $t_1$ is odd (and hence $t_1-1$ is even) or $t_2$ is even and thus $\big\lceil\frac{3(t_1-1)}{2}\big\rceil+\big\lceil\frac{3t_2}{2}\big\rceil\leq \frac{3(t_1-1)}{2}+\frac{3t_2}{2}+\frac12$ holds, since the expressions $\big\lceil\frac{3(t_1-1)}{2}\big\rceil-\frac{3(t_1-1)}{2}$ and $\big\lceil\frac{3t_2}{2}\big\rceil-\frac{3t_2}{2}$ are both bounded by $\frac12$ and at least one of these two expressions is equal to $0$. Therefore we obtain
\[\chi(G)\leq \left\lceil\frac{3(t_1-1)}{2}\right\rceil +\left\lceil\frac{3t_{2}}{2}\right\rceil+1\leq \frac{3(t_1-1)}{2}+\frac{3t_2}{2}+\frac12+1=\frac{3(t_1+t_2)}{2} \leq\left\lceil\frac{3(t-1)}{2}\right\rceil.\]
In any case, this yields a contradiction to the assumption $\chi(G)>\left\lceil\frac{3}{2}(t-1)\right\rceil$ and concludes the proof of Case 1.

\noindent\textbf{{Case 2: For all $i\in \{1,2\}$ and $y\in Y_i$, the set $(X_i\setminus Y_i)\cup \{y\}$ is not independent.}}

This, in particular, implies that $X_1\setminus Y_1\neq \varnothing$ and $X_2\setminus Y_2\neq \varnothing$.
Let $T\subseteq G[(X_1\setminus Y_1)\cup (X_2\setminus Y_2)]$ be a tree with a proper $2$-coloring defined as follows. If 
$X_1\setminus Y_1$ is an independent set in $G$, then we let $T$ be any spanning tree of the complete bipartite graph spanned between $X_1\setminus Y_1$ and $X_2\setminus Y_2$ in $G$, and we color all vertices in $X_1\setminus Y_1$ white and all vertices in $X_2\setminus Y_2$ black. On the other hand, if the set $X_1\setminus Y_1$ is not independent, then we pick two adjacent vertices $u,v\in X_1\setminus Y_1$, and some vertex $w\in X_{2}\setminus Y_{2}$, and define $T$ as the path with vertices $u,v,w$ in $G$ (in this order), with $u$ and $w$ colored white and $v$ colored black. 

We claim that now, every vertex in $Y_1\cup Y_2$ has a white vertex in $T$ as a neighbor. To verify this, consider first the case that $X_1\setminus Y_1$ is independent, then the set of white vertices in $T$ is $X_1\setminus Y_1$. Now, by assumption of Case 2, every vertex $y\in Y_1$ must have a neighbor in $X_1\setminus Y_1$ (i.e., a white neighbor in $T$). Furthermore, every vertex $y\in Y_2\su X_2$ also has a neighbor in $X_{1}\setminus Y_{1}$ (since all possible edges between $X_1$ and $X_2$ exist in $G$). Thus, every vertex in $Y_1\cup Y_2$ has a white vertex in $T$ as a neighbor if $X_1\setminus Y_1$ is independent.

Next, consider the second case, where $T$ is the path on vertices $u,v,w$ with $u,v\in X_1$ and $w\in X_{2}$, and $u,w$ are the white vertices of $T$. Now, every vertex $y\in Y_1\su X_1$ is a neighbor of $w$ and every vertex $y\in Y_{2}\su X_{2}$ is a neighbor of $u$. Thus, again every vertex in $Y_1\cup Y_2$ has a white vertex in $T$ as a neighbor.

Now, the collection of the trees $T_1^1,\ldots,T_{t_1}^1, T_1^2,\ldots,T_{t_2}^2, T$ shows that $G$ contains  $K_{t_1+t_2+1}$ as an odd minor in $G$. Indeed, we clearly have a monochromatic edge between any two of the trees $T_1^1,\ldots,T_{t_1}^1$ as well as between any two of the trees $T_1^2,\ldots,T_{t_2}^2$, by how we chose these trees. Furthermore, for any $j_1\in [t_1]$ and $j_2\in [t_2]$ there is a monochromatic edge between the trees $T_{j_1}^1$ and $T_{j_2}^2$, since both of them contain a white vertex, and these two white vertices form an edge in $G$ (since $G$ contains all possible edges between $X_1$ and $X_2$). Finally, every vertex in $Y_1\cup Y_2$ has a white neighbor in $T$, so in particular from each of the trees $T_1^1,\ldots,T_{t_1}^1, T_1^2,\ldots,T_{t_2}^2$ there is a monochromatic edge to~$T$.

Since we assumed that $G$ does not contain $K_t$ as an odd minor, we can now conclude that $t_1+t_2+1 \leq t-1$. This implies
\[
\chi(G)=\chi(G_1)+\chi(G_2) \leq \left\lceil\frac{3t_1}{2}\right\rceil+\left\lceil\frac{3t_2}{2}\right\rceil \leq \left(\frac{3t_1}{2}+\frac{1}{2}\right)+\left(\frac{3t_2}{2}+\frac{1}{2}\right)<\frac{3(t_1+t_2+1)}{2} \leq \frac{3}{2}(t-1),
\]
contradicting our initial assumption $\chi(G)>\left\lceil\frac{3}{2}(t-1)\right\rceil$. This concludes the proof of Case~2 and thereby of the whole proposition. 
\end{proof}

\end{document}